\documentclass[twoside,11pt,hidelinks]{amsart}

\usepackage[utf8]{inputenc}
\usepackage[english]{babel}
\usepackage[fixlanguage]{babelbib}
\usepackage{amssymb, amsfonts, amsmath, amsthm, amscd, mathtools, thmtools, tikz, xcolor, hyperref, a4wide}
\usepackage[numbers,square]{natbib}
\usepackage[all]{xy}
\usepackage{ xfrac, calc, indentfirst, cleveref, aligned-overset, easyReview, enumitem, float, tikz-cd, microtype, tensor, epsfig, stackengine, scalerel, wasysym, color, graphicx }

\newcounter{contador}
\numberwithin{contador}{section}

\newtheorem{theorem}[contador]{Theorem}
\newtheorem{prop}[contador]{Proposition}
\newtheorem{lemma}[contador]{Lemma}
\newtheorem{corollary}[contador]{Corollary}

\theoremstyle{definition}
\newtheorem{defi}[contador]{Definition}
\newtheorem{obs}[contador]{Remark}
\newtheorem{exe}[contador]{Example}

 \newcommand{\Not}[2][0]{	
	\setcounter{enumi}{#1}
	\renewcommand{\theenumi}{#2\arabic{enumi}}
	\renewcommand{\labelenumi}{(\theenumi)}
	\setlength{\itemindent}{\widthof{#2}}
	\setlength{\itemsep}{4pt}
}

\allowdisplaybreaks 

\title[A Generalization of the Ehresmann-Schein-Nambooripad Theorem]{A Generalization of the Ehresmann-Schein-Nambooripad Theorem to Two-Sided Ehresmann Semigroupoids}
\author[Haag and Tamusiunas]{Rafael Haag$^*$ and Thaísa Tamusiunas}
\address{Instituto de Matem\'{a}tica, Universidade Federal do Rio Grande do Sul,  Av. Bento Gon\c{c}alves, 9500, 91509-900. Porto Alegre-RS, Brazil}
\email{rafaelpetasny@gmail.com}
\email{thaisa.tamusiunas@gmail.com}
\thanks{$^*$ Corresponding author}
\date{} 

\begin{document}

    
    \begin{abstract}
        We introduce the notion of two-sided Ehresmann semigroupoids and show that they are in correspondence with a specific class of categories, which we call local biordered Ehresmann categories. This correspondence provides a unified generalization of the Ehresmann-Schein-Nambooripad Theorem for both inverse semigroupoids and Ehresmann semigroups. In particular, two-sided restriction semigroupoids form a distinguished subclass of two-sided Ehresmann semigroupoids, and for this case we describe the associated class of categories, extending earlier results for restriction semigroups.
    \end{abstract}

    \maketitle

\section{Introduction}

        Pseudogroups arose in the study of symmetries of differential equations as a generalization of transformation groups. Such structures consist of collections of homeomorphisms between open subsets of a topological space satisfying certain conditions. In an effort to obtain an abstract framework that includes pseudogroups, two different solutions appeared: inverse semigroups, introduced independently by V. Wagner \cite{wagner1952semigroups} and G. Preston \cite{preston1954semigroups}, and inductive groupoids, introduced by C. Ehresmann \cite{ehresmann1960inductive}.

        The classical Ehresmann-Schein-Nambooripad Theorem (ESN Theorem for short) shows that Wagner and Preston's approach to pseudogroups is, in fact, equivalent to Ehresmann's. More precisely, this theorem provides a constructive algorithm for obtaining an inductive groupoid $\mathbf{G}(S)$ from an inverse semigroup $S$, and for obtaining an inverse semigroup $\mathbf{S}(G)$ from an inductive groupoid $G$. These constructions are mutually inverse, in the sense that $\mathbf{G}(\mathbf{S}(G)) = G$ as inductive semigroupoids and $\mathbf{S}(\mathbf{G}(S)) = S$ as inverse semigroups. Furthermore, it shows that a function $\varphi \colon S \to T$ between inverse semigroups is a semigroup homomorphism if and only if the corresponding function $\varphi \colon \mathbf{G}(S) \to \mathbf{G}(T)$ is an appropriate structure-preserving functor. This correspondence establishes an isomorphism between the category of inverse semigroups and the category of inductive groupoids. The ESN Theorem has been generalized to many classes of semigroups (see \cite{armstrong1988}, \cite{fitz2021}, \cite{gould2009restriction}, \cite{gould2011actions}, \cite{rsgpdexpansion}, \cite{hollings2010extending}, \cite{lawson1991}, \cite{nambooripad1974}, \cite{stokes2017d}, \cite{stokes2023}, \cite{wang2019}, \cite{wang2020}, \cite{wang2022}, \cite{wang2023d}, \cite{wang2014} and \cite{wang2016}) and to inverse semigroupoids (see \cite{dewolf2018ehresmann}). Among these, three generalizations of the ESN Theorem are particularly relevant to this work.

        In \cite[Theorem 3.17]{lawson1991}, Mark Lawson studies the class of semigroups $S$ endowed with two unary operations, denoted $\ast,+ \colon S \to S$, such that the partially defined operation
            $$ s \circ t = st, \ \text{ defined whenever } s^\ast = t^+, $$
        endows the set $S$ with a category structure. This structure is enriched with additional operations that allow one to recover the original semigroup multiplication. Such semigroups were called Ehresmann semigroups, and the corresponding categories were called Ehresmann categories. In the same work, Lawson extended the classical ESN Theorem to an isomorphism between the category of Ehresmann semigroups and the category of Ehresmann categories \cite[Theorem 4.24]{lawson1991}. As a particular case, Lawson proved that the category of two-sided restriction semigroups (also called \emph{idempotent connected Ehresmann semigroups}, or \emph{weakly $E$-ample semigroups}) is isomorphic to the category of inductive categories \cite[Theorem 5.7]{lawson1991}. 
        
        The latter result was generalized by Christopher Hollings to two further category isomorphisms, in which the objects remain the same but the morphisms are weakened in two dual ways, namely via $\vee$-premorphisms and via $\wedge$-premorphisms. This constitutes the second reference that is particularly relevant to our work. More precisely, Hollings proved that $\vee$-premorphisms between two-sided restriction semigroups correspond to ordered functors between inductive categories \cite[Theorem 4.1]{hollings2010extending}, and that $\wedge$-premorphisms between two-sided restriction semigroups correspond to inductive prefunctors between inductive categories \cite[Theorem 5.1]{hollings2010extending}.

        To emphasize the relevance of Hollings’ work, we note that many ESN-type theorems establish equivalences between categories whose morphisms are required to preserve the underlying structure strictly. Hollings showed that, in the context of two-sided restriction semigroups, one can go beyond this setting and obtain category isomorphisms involving more general notions of morphism. In addition to \cite{hollings2010extending}, isomorphisms of categories with $\vee$-premorphisms are also considered in the classical ESN Theorem for inverse semigroups \cite[Theorem 4.1.8]{lawson1998inverse}, its generalization to inverse semigroupoids \cite[Theorem 3.21]{dewolf2018ehresmann} and to locally restriction $P$-restriction semigroups \cite[Theorem 5.7]{wang2019}. On the other hand, isomorphisms of categories with $\wedge$-premorphisms are also considered in the ESN-type Theorems for one-sided restriction semigroups \cite[Theorem 6.6]{gould2011actions} and for one-sided restriction semigroupoids \cite[Theorem 6.5]{lrspgesn}. Both types of premorphisms were introduced by Donald McAlister and Norman Reilly in their study of $E$-unitary covers for inverse semigroups \cite{mcalister1977}. More recently, $\wedge$-premorphisms have become relevant due to their connection to partial actions.

        Lastly, Darien DeWolf and Dorette Pronk generalized the classical ESN Theorem to inverse semigroupoids \cite[Corollary 3.18]{dewolf2018ehresmann}. In this context, the groupoids corresponding to inverse semigroupoids were called locally inductive groupoids. They also characterized the property satisfied by locally inductive groupoids that allow the corresponding inverse semigroupoid to be endowed with a category structure \cite[Theorem 3.16]{dewolf2018ehresmann}. 
        
        Taking into account the results of Lawson, DeWolf and Pronk, and Hollings, the following questions naturally arise.
       
        \vspace{0,1cm}
        
        \emph{Since Ehresmann semigroups and inverse semigroupoids generalize inverse semigroups in different directions, is there a broader algebraic framework that unifies these constructions and supports an ESN-type correspondence? Furthermore, in which level of generality can we obtain isomorphisms of categories involving $\wedge$-premorphisms and $\vee$-premorphisms?}

\vspace{0,1cm}

        Motivated by these questions, the aim of this paper is to introduce a class of algebraic structures that encompasses both Ehresmann semigroups and inverse semigroupoids, which we call \emph{Ehresmann semigroupoids}. We then establish an ESN-type correspondence between these semigroupoids and a suitable class of categories, referred to as \emph{local biordered Ehresmann categories}. We further investigate how this correspondence behaves when the notion of structure-preserving morphism is extended to $\vee$-premorphisms and $\wedge$-premorphisms. We show that a correspondence at the level of $\vee$-premorphisms and $\wedge$-premorphisms is obtained within the class of two-sided restriction semigroupoids. For this reason, the final two subsections of the paper are devoted to restriction semigroupoids.

        This paper is structured as follows. In Section 2, we introduce Ehresmann semigroupoids, local biordered Ehresmann categories, and the machinery required to establish a bijective ESN-type correspondence between them. In Section 3, we restrict the ESN-type correspondence to the class of restriction semigroupoids and the class of inverse semigroupoids, and describe the corresponding classes of categories associated with each type of semigroupoid. In particular, we investigate the conditions on the associated category under which the semigroupoid admits a category structure or a semigroup structure. Finally, in Section 4, we obtain three category isomorphisms: the category of Ehresmann semigroupoids and (2,1,1)-morphisms is isomorphic to the category of local biordered Ehresmann categories and inductive functors; the category of restriction semigroupoids and $\vee$-premorphisms is isomorphic to the category of inductive categories and ordered functors; and the category of restriction semigroupoids and $\wedge$-premorphisms is isomorphic to the category of inductive categories and inductive prefunctors.

\section{Relating Semigroupoids and Categories}

    In this section, we introduce the class of Ehresmann semigroupoids, together with the class of local biordered Ehresmann categories. To each Ehresmann semigroupoid $S$ we associate a local biordered Ehresmann category $C(S)$, and to each local biordered Ehresmann category $\mathcal{C}$ we associate an Ehresmann semigroupoid $S(\mathcal{C})$. The main result of this section is that $S(C(S)) = S$ as Ehresmann semigroupoids and that $C(S(\mathcal{C})) = \mathcal{C}$ as local biordered Ehresmann categories. This correspondence is the \textit{object part} of three category isomorphisms that we develop in Section 4.

\subsection{Ehresmann Semigroupoids} In this subsection, we introduce one-sided and two-sided Ehresmann semigroupoids, which generalize both restriction semigroupoids \cite{rsgpdexpansion} and Ehresmann semigroups \cite{lawson1991}. To each left (respectively, right) Ehresmann semigroupoid, we associate a partial order $\leq_l$ (respectively, $\leq_r$). We begin with the definition of a semigroupoid.

    \begin{defi} \label{def:semigroupoid} \cite[Definition 2.1]{exel2011semigroupoid}
        A \emph{semigroupoid} is a triple $(S,S^{(2)},\cdot)$ such that $S$ is a set, $S^{(2)}$ is a subset of $S \times S$, and $\cdot \colon S^{(2)} \to S$ is an operation which is associative in the following sense: if $r,s,t \in S$ are such that either
        \begin{enumerate} \Not{s}
            \item $(s,t) \in S^{(2)}$ and $(t,r) \in S^{(2)}$, or \label{s1}
            \item $(s,t) \in S^{(2)}$ and $(s \cdot t,r) \in S^{(2)}$, or \label{s2}
            \item $(t,r) \in S^{(2)}$ and $(s,t \cdot r) \in S^{(2)}$, \label{s3}
        \end{enumerate}
        then all of $(s,t)$, $(t,r)$, $(s \cdot t,r)$ and $(s,t \cdot r)$ lie in $S^{(2)}$, and $(s \cdot t) \cdot r = s \cdot (t \cdot r)$.
    \end{defi}

    Semigroupoids in the sense of the preceding definition are sometimes called \emph{Exel semigroupoids}. We emphasize that there is another, older notion of semigroupoid, namely, a directed graph endowed with a partial binary operation. Such semigroupoids were introduced by Tilson \cite{tilson1987categories} and are also known as \emph{Tilson semigroupoids}, \emph{semicategories}, or \emph{categorical semigroupoids}. It was shown in \cite[Theorem 2.15]{cordeiro2023etale} that the class of Exel semigroupoids properly contains the class of Tilson semigroupoids. Throughout this paper, the term \emph{semigroupoid} always refers to an Exel semigroupoid.

    For simplicity, we denote a semigroupoid by $S$ rather than by the triple $(S, S^{(2)}, \cdot)$. The operation $\cdot$ will be referred to as the \emph{composition} and will be written by juxtaposition; that is, $st = s \cdot t$. We say that $st$ \emph{is defined} whenever $(s,t) \in S^{(2)}$.

    \begin{defi}
        A \emph{left Ehresmann semigroupoid} is a pair $(S,+)$, where $S$ is a semigroupoid and $+ \colon S \to S$ is a unary operation satisfying the following conditions:
        \begin{enumerate} \Not{le}
            \item for every $s \in S$, $s^+s$ is defined and $s^+s = s$; \label{le1}
            \item if $s^+t^+$ is defined, then $t^+s^+$ is defined and $s^+t^+ = t^+s^+$; \label{le2}
            \item if $s^+t^+$ is defined, then $s^+t^+ = (s^+t)^+$; \label{le3}
            \item if $st$ is defined, then $(st)^+ = (st^+)^+$. \label{le4}
        \end{enumerate}
        In this case, we denote by $S^+ = \{ s^+ \colon s \in S \}$ the \emph{set of projections} of $(S,+)$.
    \end{defi}

    We note that left Ehresmann semigroups, introduced in \cite[p. 441]{lawson1991}, are precisely those left Ehresmann semigroupoids $(S,+)$ whose binary operation is total, that is, for which $S^{(2)}=S\times S$.

    \begin{obs} \label{obs:le}
        Axioms \eqref{le3} and \eqref{le4} are well defined. In fact, it follows from \eqref{le1} and the associativity of $S$ that $st = s(t^+t)$ is defined if and only if $st^+$ is defined. Consequently, if $s^+t^+$ is defined, then so is $s^+t$ in \eqref{le3}, and if $st$ is defined, then so is $st^+$ in \eqref{le4}.
    \end{obs}

    \begin{lemma} \label{lema:le}
        Let $(S,+)$ be a left Ehresmann semigroupoid. Then the following statements hold.
        \begin{itemize}
            \item[(a)] For every $e \in S^+$, $ee$ is defined, $ee=e$ and $e^+ = e$.
            \item[(b)] If $st$ is defined, then $(st)^+ = (st)^+s^+$.
        \end{itemize}

        \begin{proof}
            (a) If $e \in S^+$, then $e = s^+$ for some $s \in S$. Since $s^+s$ is defined by \eqref{le1}, it follows from Remark \ref{obs:le} that $ee = s^+s^+$ is defined. From \eqref{le3} and \eqref{le1}, we obtain $ee = s^+s^+ = (s^+s)^+ = s^+ = e$. On the other hand, we have
            \begin{align*}
                e = s^+ \overset{\eqref{le1}}{=} (s^+)^+ s^+ \overset{\eqref{le2}}{=} s^+ (s^+)^+ \overset{\eqref{le3}}{=} (s^+s^+)^+ = (s^+)^+ = e^+.
            \end{align*}

            (b) From \eqref{le1}, we have $st = s^+st$. In particular, $s^+(st)$ is defined. By Remark \ref{obs:le}, it follows that $s^+(st)^+$ is defined. Using \eqref{le3}, we obtain $s^+(st)^+ = (s^+st)^+ = (st)^+$. Finally, by \eqref{le2}, we have $(st)^+ = (st)^+s^+$.
        \end{proof}
    \end{lemma}

    \begin{lemma} \label{lema:le-order}
       Let $(S,+)$ be a left Ehresmann semigroupoid. For $s,t \in S$, the following conditions are equivalent.
    \begin{itemize}
        \item[(a)] The product $s^+ t$ is defined and satisfies $s^+ t = s$.
        \item[(b)] There exists $e \in S^+$ such that $e t$ is defined and $e t = s$.
    \end{itemize}
    Moreover, the relation $\leq_l$ on $S$, defined by
\[
s \leq_l t \quad \text{if and only if condition (a) holds},
\]
is a partial order on $S$.

        \begin{proof}
            If (a) holds, then (b) holds with $e = s^+$. Conversely, suppose that (b) holds. Since $e = e^+$ by Lemma \ref{lema:le}(a), we have
                $$ s^+ = (et)^+ = (e^+t)^+ \overset{\eqref{le3}}{=} e^+t^+ = et^+. $$
            Using \eqref{le1} together with the previous equality, we obtain $s = et = et^+t = s^+t$. This shows that conditions (a) and (b) are equivalent. 
            
            We now prove that $\leq_l$ is a partial order on $S$. From \eqref{le1}, we have $s \leq_l s$ for all $s \in S$, and hence $\leq_l$ is reflexive. If $s \leq_l t$ and $t \leq_l s$, then $s^+t = s$ and $t^+s = t$. Therefore,
                $$ s = s^+t = s^+t^+s \overset{\eqref{le2}}{=} t^+s^+s \overset{\eqref{le1}}{=} t^+s = t, $$
            which shows that $\leq_l$ is antisymmetric. Finally, if $s \leq_l t$ and $t \leq_l r$, then
                $$ s = s^+t = s^+t^+r \overset{\eqref{le3}}{=} (s^+t)^+ r. $$
            Since $(s^+t)^+ \in S^+$, condition (b) implies that $s^+r$ is defined and that $s^+r = s$. Therefore, $\leq_l$ is transitive. This proves that $\leq_l$ is a partial order on $S$.
        \end{proof}
    \end{lemma}

    In an analogous manner, we define a right Ehresmann semigroupoid as follows.

    \begin{defi}
        A \emph{right Ehresmann semigroupoid} is a pair $(S,\ast)$, where $S$ is a semigroupoid and $\ast \colon S \to S$ is a unary operation satisfying the following conditions:
        \begin{enumerate} \Not{re}
            \item for every $s \in S$, $ss^\ast$ is defined and $ss^\ast = s$; \label{re1}
            \item if $s^\ast t^\ast$ is defined, then $t^\ast s^\ast$ is defined and $s^\ast t^\ast = t^\ast s^\ast$; \label{re2}
            \item if $s^\ast t^\ast$ is defined, then $s^\ast t^\ast = (st^\ast)^\ast$; \label{re3}
            \item if $st$ is defined, then $(st)^\ast = (s^\ast t)^\ast$. \label{re4}
        \end{enumerate}
        In this case, we denote by $S^\ast = \{s^\ast \colon s \in S\}$ the \emph{set of projections} of $(S,\ast)$.
    \end{defi}

    Right Ehresmann semigroupoids are dual to left Ehresmann semigroupoids in the following sense. Let $S$ be a semigroupoid. Define a new operation on $S$ by $s \cdot^{op} t = t \cdot s$, whenever $t \cdot s$ is defined. Then $(S,\cdot,+)$ is a left Ehresmann semigroupoid if and only if $(S,\cdot^{op},+)$ is a right Ehresmann semigroupoid. As a consequence, axioms \eqref{re3} and \eqref{re4} are well defined, and the dual statements of Lemmas \ref{lema:le} and \ref{lema:le-order} hold.

    We now introduce two-sided Ehresmann semigroupoids, starting with a lemma.

    \begin{lemma}
      Let $S$ be a semigroupoid, and let $+,\ast \colon S \to S$ be unary operations on $S$. Suppose that $(S,+)$ is a left Ehresmann semigroupoid and that $(S,\ast)$ is a right Ehresmann semigroupoid. Then the following conditions are equivalent.
\begin{itemize}
\item[(a)] The sets of projections $S^+$ and $S^\ast$ coincide.
\item[(b)] For all $s \in S$, we have $s^+ = (s^+)^\ast$ and $s^\ast = (s^\ast)^+$.
\end{itemize}

        \begin{proof}
            Suppose that $S^+ = S^\ast$. Then, by Lemma \ref{lema:le}(a) and its dual, we obtain $s^+ = (s^+)^\ast$ and $s^\ast = (s^\ast)^+$ for every $s \in S$. Conversely, suppose that $s^+ = (s^+)^\ast$ and $s^\ast = (s^\ast)^+$ for all $s \in S$. Then, for every $e \in S^+$, there exists $s \in S$ such that $e = s^+ = (s^+)^\ast \in S^\ast$ and hence $S^+ \subseteq S^\ast$. By a symmetric argument, we also obtain $S^\ast \subseteq S^+$. Therefore, $S^+ = S^\ast$.
        \end{proof}
    \end{lemma}

    \begin{defi} \label{defi:two-sided-Ehresmann-semigroupoid}
        A \emph{two-sided Ehresmann semigroupoid} is a triple $(S,+,\ast)$, where $(S,+)$ is a left Ehresmann semigroupoid, $(S,\ast)$ is a right Ehresmann semigroupoid, and, for every $s \in S$, the following identities hold:
        \begin{align*}
            s^+ = (s^+)^\ast \quad\text{and}\quad s^\ast = (s^\ast)^+. \tag{E} \label{lre}
        \end{align*}
        In this case, we denote by $U = S^+ = S^\ast$ the set of projections of $(S,+,\ast)$.
    \end{defi}

    Natural examples of Ehresmann semigroupoids arise as graphs of actions of Ehresmann semigroups. Recall that an action of a group $G$ on a set $X$ is a group morphism $\beta$ from $G$ to the group of bijections of $X$. The graph associated with such an action was defined in \cite[Definition 2.8]{exel2014book} and carries a natural groupoid structure.

    Actions of two-sided restriction semigroups are considered in \cite[p. 21]{machado2024}. Extending this notion to Ehresmann semigroups, we define an action of an Ehresmann semigroup $(S,+,\ast)$ on a set $X$ to be a semigroup morphism $\beta$ from $S$ to $PT(X)$, where $PT(X)$ denotes the semigroup of partial transformations of $X$, such that $\beta_s \colon {_{s^\ast}X} \to {_{s^+}X}$, for all $s \in S$, and $\beta_e = id_{_eX}$, for all $e \in U$.
   
    \begin{exe}
        Let $(S,+,\ast)$ be an Ehresmann semigroup, $X$ be a nonempty set and $\beta$ be an action of $(S,+,\ast)$ on $X$. The graph of the action $\beta$ is the directed graph $\mathcal{G}_\beta = (\mathcal{G}_0,\mathcal{G}_1,D,R)$, where
        \begin{gather*}
            \mathcal{G}_0 = X, \quad \mathcal{G}_1 = \{ (\beta_s(x),s,x) \colon x \in {_{s^\ast}X} \}, \\
            D(\beta_s(x),s,x) = x \quad\text{and}\quad R(\beta_s(x),s,x) = \beta_s(x).
        \end{gather*}
        This directed graph can be endowed with an Ehresmann semigroupoid structure induced by the action. The composition $(\beta_t(y),t,y) \cdot (\beta_s(x),s,x)$ is defined if and only if $y = \beta_s(x)$, in which case it is given by $(\beta_{ts}(x),ts,x)$, and the unary operations are
            $$ (\beta_s(x),s,x)^\ast = (x,s^\ast,x) \quad\text{and}\quad (\beta_s(x),s,x)^+ = (\beta_s(x),s^+,\beta_s(x)). $$
        Note that $\mathcal{G}_\beta$ is a semigroup if and only if the set $\bigcup_{e \in U} {_eX}$ has a single element.
    \end{exe}

    In this paper, we develop a generalization of the ESN Theorem for the class of two-sided Ehresmann semigroupoids and some of its subclasses. A one-sided version of the ESN Theorem has already been obtained in the special case of left restriction semigroupoids \cite{lrspgesn}. Therefore, from this point on, the term “Ehresmann semigroupoid” will always refer to a two-sided Ehresmann semigroupoid.

    \begin{lemma} \label{lema:lre-order}
        Let $(S,+,\ast)$ be an Ehresmann semigroupoid. Then the partial orders $\leq_l$ and $\leq_r$ coincide on the set of projections $U$.

        \begin{proof}
            Let $e,f \in U$. Then $e = e^+ = e^\ast$ by Lemma \ref{lema:le}(b) and its dual. Moreover, by \eqref{le2}, $ef$ is defined if and only if $fe$ is defined, and in this case $ef = fe$. Therefore, $e \leq_l f$ if and only if $ef = fe$ is defined and $ef = e = fe$, which holds if and only if $e \leq_r f$
        \end{proof}
    \end{lemma}

    To conclude this subsection, we present two classes of Ehresmann semigroupoids that will be useful later, namely, semigroups of relations on a set and local meet-semilattices. The semigroup of relations on a set $X$ was introduced in \cite{zaretskii1962}. It was later observed that this semigroup admits a natural structure of an Ehresmann semigroup (see, for instance, \cite{jackson2001}).

    \begin{exe} \label{exe:relX}
        Let $X$ be a set. A relation on $X$ is a subset $R \subseteq X \times X$, and we denote by $\mathcal{B}(X)$ the set of all relations on $X$. Then $\mathcal{B}(X)$ carries a natural Ehresmann semigroup structure, given by
            $$ R \circ S = \{ (x,y) \in X \times X \colon (x,z) \in R \text{ and } (z,y) \in S, \text{ for some } z \in X \}, $$
            $$ R^+ = \{ (x,x) \in X \times X \colon (x,y) \in R \} \quad\text{and}\quad R^\ast = \{ (y,y) \in X \times X \colon (x,y) \in R \}. $$
    \end{exe}

    A \textit{meet-semilattice} is a partially ordered set $(X, \leq)$ such that every pair of elements $x,y \in X$ has a \textit{greatest lower bound} in $X$, denoted by $x \wedge y$. In this case, the assignment $(x,y) \mapsto x \wedge y$ is called the \textit{meet} operation of $(X,\leq)$. In what follows, we introduce a generalization of meet-semilattices as a special class of semigroupoids and then characterize these semigroupoids in terms of partially ordered sets.

    \begin{defi} \label{def:localmeet}
        A \emph{local meet-semilattice} is a semigroupoid $S$ such that, for every $e \in S$, the product $ee$ is defined and satisfies $ee = e$, and whenever $ef$ is defined, the product $fe$ is also defined and $ef = fe$. That is, the operation in $S$ is commutative whenever it is defined.
    \end{defi}

    \begin{exe} \label{exe:localmeet}
        Every local meet-semilattice $S$ can be regarded as an Ehresmann semigroupoid with additional structure defined by $s^+ = s$ and $s^\ast = s$, for all $s \in S$. The verification of axioms \eqref{le1}, \eqref{le2}, \eqref{le3}, and \eqref{le4} is straightforward. Since a local meet-semilattice is commutative, the dual axioms are also satisfied. In this case, we have $S^+ = S = S^\ast$.
    \end{exe}

    \begin{prop} \label{prop:localmeet}
        There is a bijective correspondence between:
        \begin{itemize}
            \item[(a)] meet-semilattices (in the usual sense) and local meet-semilattices that are semigroups;
            \item[(b)] local meet-semilattices and disjoint unions of meet-semilattices.
        \end{itemize}

        \begin{proof}
            (a) If $(X,\leq,\wedge)$ is a meet-semilattice, then $\wedge \colon X \times X \to X$ is an associative and commutative operation such that $x \wedge x = x$, for all $x \in X$. Hence, $(X,\wedge)$ is a semigroup and a local meet-semilattice.
            
            Conversely, let $S$ be a local meet-semilattice that is also a semigroup. From Example \ref{exe:localmeet}, we have $U = S$. Hence, it follows from Lemma \ref{lema:lre-order} that the partial orders $\leq_l$ and $\leq_r$ coincide over $S$. Let $\leq = \leq_l = \leq_r$. Then $s \leq t$ if and only if $st = s$. Now, it is easy to see that $st$ is the greatest lower bound for $s$ and $t$. Thus, $(S,\leq)$ is a meet-semilattice with meet operation given by $s \wedge t = st$. This concludes that meet-semilattices (in the usual sense) are precisely the local meet-semilattices (as in Definition \ref{def:localmeet}) that are also semigroups.\\

            \noindent(b) Let $(X_i,\leq_i,\wedge_i){i \in I}$ be a family of pairwise disjoint meet-semilattices. By (a), each $(X_i,\wedge_i)$ is a commutative semigroup such that $e \wedge_i e = e$ for all $e \in X_i$. Let $X = \bigcup_{i \in I} X_i$, and define a partial operation on $X$ by declaring that $xy$ is defined if and only if $x,y \in X_i$ for some $i \in I$, in which case $xy = x \wedge_i y$. Then $X$ is a commutative semigroupoid such that $ee = e$ for all $e \in X$. Hence, $X$ is a local meet-semilattice.

            Conversely, let $S$ be a local meet-semilattice and define a relation $\omega$ on $S$ by declaring that $s \omega t$ if and only if $st$ is defined. The relation $\omega$ is reflexive, since for every $e \in S$ the product $ee$ is defined; symmetric, since the operation in $S$ is commutative; and transitive, since
            \begin{align*}
                st, tr \text{ are defined} &\iff st, (st)r \text{ are defined} \\
                &\iff ts, (ts)r \text{ are defined} \iff ts, sr \text{ are defined}.
            \end{align*}
            Hence, $\omega$ is an equivalence relation, and $S$ is the disjoint union of its equivalence classes. For each $e \in S$, the equivalence class $\omega(e) = \{ f \in S \colon ef \text{ is defined} \}$ of $\omega$ is a meet-semilattice. Indeed, if $f,g \in \omega(e)$, then $fg$ is defined and $e(fg)$ is also defined, so that $fg \in \omega(e)$. Since $S$ is a local meet-semilattice and the operation restricted to each equivalence class $\omega(e)$ is totally defined, it follows from (a) that $\omega(e)$ is a meet-semilattice. This shows that every local meet-semilattice is a disjoint union of meet-semilattices..
        \end{proof}
    \end{prop}

In many situations, it is convenient to view local meet-semilattices as semigroupoids. On the other hand, Proposition \ref{prop:localmeet} shows that the operation on a local meet-semilattice is completely determined by its partial order. In this case, we may denote the operation of a local meet-semilattice by $\wedge$.

\subsection{Local Ehresmann Categories} In this subsection, we fix our conventions for categories, recall the notion of an ordered category, and introduce the class of local biordered Ehresmann categories, which generalize both Ehresmann categories \cite{lawson1991} and locally inductive groupoids \cite{dewolf2018ehresmann}.

    \begin{defi}
        A \emph{small category} is a quintuple $(\mathcal{C},\mathcal{C}_0,D,R,\circ)$, where $\mathcal{C}_0 \subseteq \mathcal{C}$ are sets, $D,R \colon \mathcal{C} \to \mathcal{C}_0$ are functions, and $\circ \colon \mathcal{C}^{(2)} \to \mathcal{C}$ is a binary operation, where
            $$ \mathcal{C}^{(2)} = \{ (x,y) \in \mathcal{C} \times \mathcal{C} \colon D(x) = R(y) \}, $$
        satisfying the following conditions:
        \begin{enumerate} \Not{C}
            \item if $D(x) = R(y)$, then $D(x \circ y) = D(y)$ and $R(x \circ y) = R(x)$; \label{C1}
            \item if $D(x) = R(y)$ and $D(y) = R(z)$, then $x \circ (y \circ z) = (x \circ y) \circ z$; \label{C2}
            \item for every $e \in \mathcal{C}_0$, we have $D(e) = e = R(e)$, and whenever $D(x) = e = R(y)$, it holds that $x \circ e = x$ and $e \circ y = y$. \label{C3}
        \end{enumerate}
    \end{defi}

    In the present work, we only consider small categories. Thus, we omit the term \emph{small} and simply write \emph{category}. Every category $(\mathcal{C},\mathcal{C}_0,D,R,\circ)$ can be viewed as a semigroupoid $(\mathcal{C},\mathcal{C}^{(2)},\circ)$, with $\mathcal{C}^{(2)}$ defined as above. Accordingly, we shall simply denote a category by $\mathcal{C}$ and say that “$x \circ y$ is defined” instead of writing $(x,y) \in \mathcal{C}^{(2)}$ or $D(x) = R(y)$. Moreover, every category can be regarded as an Ehresmann semigroupoid $(\mathcal{C},+,\ast)$ by defining $x^+ = R(x)$ and $x^\ast = D(x)$ for every $x \in \mathcal{C}$. In this case, the partial orders $\leq_l$ and $\leq_r$ on $\mathcal{C}$ are trivial; that is, $x \leq_l y$ if and only if $x = y$, which is equivalent to $x \leq_r y$.

    \begin{defi}
        An \emph{ordered category} is a pair $(\mathcal{C},\leq)$, where $\mathcal{C}$ is a category and $\leq$ is a partial order on $\mathcal{C}$ satisfying the following conditions:
        \begin{enumerate} \Not{O}
            \item if $x \leq y$, then $D(x) \leq D(y)$ and $R(x) \leq R(y)$; \label{O1}
            \item if $x \leq y$, $x' \leq y'$ and $x \circ x'$ and $y \circ y'$ are defined, then $x \circ x' \leq y \circ y'$. \label{O2}
        \end{enumerate}
        We say that an ordered category has \emph{restrictions} if it satisfies the following condition:
        \begin{align*}
            e \in \mathcal{C}_0,\ e \leq D(x) \Longrightarrow \exists! \, x|e \in \mathcal{C}
        \text{ such that } x|e \leq x \text{ and } D(x|e) = e. \tag{Or}\label{restriction}
        \end{align*}
        Analogously, an ordered category is said to have \emph{corestrictions} if it satisfies:
        \begin{align*}
            e \in \mathcal{C}_0,\ e \leq R(x) \Longrightarrow \exists! \, e|x \in \mathcal{C} \text{ such that } e|x \leq x \text{ and } R(e|x) = e. \tag{Oc}\label{corestriction}
        \end{align*}
        When they exist, the elements $x|e$ and $e|x$ are called the \emph{restriction} and the \emph{corestriction} of $x$ to $e$, respectively.
    \end{defi}

    \begin{lemma} \label{lema:oc-order}
        Let $(\mathcal{C},\leq)$ be an ordered category.
        \begin{itemize}
            \item[(a)] If $(\mathcal{C},\leq)$ has restrictions, then $x \leq y$ if and only if $D(x) \leq D(y)$ and $y|D(x) = x$.
            \item[(b)] If $(\mathcal{C},\leq)$ has corestrictions, then $x \leq y$ if and only if $R(x) \leq R(y)$ and $D(x)|y = x$.
        \end{itemize}

        \begin{proof}
            We prove (a). The proof of (b) is dual. Suppose that $x \leq y$. Then $D(x) \leq D(y)$ by \eqref{O1}. On the other hand, since $D(x) = D(x)$ and $x \leq y$, it follows from the uniqueness of the restriction that $y|D(x) = x$. Conversely, suppose that $D(x) \leq D(y)$ and $y|D(x) = x$. Then $x \leq y$ by the definition of restriction.
        \end{proof}
    \end{lemma}

    In the following definition, we consider $\leq_l \circ \leq_r$ and $\leq_r \circ \leq_l$ as the composition of relations in the Ehresmann semigroup $\mathcal{B}(\mathcal{C})$, as defined in Example \ref{exe:relX}.

    \begin{defi}
        A \emph{local biordered Ehresmann category} is a triple $(\mathcal{C},\leq_l,\leq_r)$ satisfying the following conditions:
        \begin{enumerate} \Not{ec}
            \item $(\mathcal{C},\leq_l)$ is an ordered category with corestrictions; \label{ec1}
            \item $(\mathcal{C},\leq_r)$ is an ordered category with restrictions; \label{ec2}
            \item \label{ec3} $\mathcal{C}_0$ is a local meet-semilattice with partial order $\leq$ and operation $\wedge$;
            \item the restriction of the partial orders $\leq_l$ and $\leq_r$ to $\mathcal{C}_0$ coincide with $\leq$; \label{ec4}
            \item the relations $\leq_l \circ \leq_r$ and $\leq_r \circ \leq_l$ coincide on $\mathcal{C}$; \label{ec5}
            \item if $x \leq_l y$ and $D(x) \wedge e$ is defined, then $x|(D(x) \wedge e) \leq_l y|(D(y) \wedge e)$; \label{ec6}
            \item if $x \leq_r y$ and $e \wedge R(x)$ is defined, then $(e \wedge R(x))|x \leq_r (e \wedge R(y))|y$. \label{ec7}
        \end{enumerate}
    \end{defi}

    \begin{obs}
        Axioms \eqref{ec6} and \eqref{ec7} are well defined. Indeed, suppose that $x \leq_l y$. Since $(\mathcal{C},\leq_l)$ is an ordered category by \eqref{ec1}, it follows from \eqref{O1} that $D(x) \leq_l D(y)$. By \eqref{ec4}, this implies that $D(x) \leq D(y)$ in $\mathcal{C}_0$. By \eqref{ec3}, $\mathcal{C}_0$ is a local meet-semilattice with partial order $\leq$, and therefore $D(x) = D(y) \wedge D(x)$. Now, if $e \wedge D(x)$ is defined, then $e \wedge (D(y) \wedge D(x))$ is also defined. By associativity of the operation $\wedge$, it follows that $e \wedge D(y)$ is defined as well. Since $(\mathcal{C},\leq_r)$ has restrictions by \eqref{ec2}, and since $e \wedge D(x) \leq D(x)$ and $e \wedge D(y) \leq D(y)$ by \eqref{ec3}, together with the fact that $\leq = \leq_r$ on $\mathcal{C}_0$ by \eqref{ec4}, we conclude that the restrictions $x|(D(x) \wedge e)$ and $y|(D(y) \wedge e)$ exist. An analogous argument shows that, if $e \wedge R(x)$ is defined, then the corestrictions $(e \wedge R(x))|x$ and $(e \wedge R(y))|y$ are also defined.
    \end{obs}
    
    Throughout the remainder of this subsection, $\mathcal{C}$ will denote a local biordered Ehresmann category $(\mathcal{C},\leq_l,\leq_r)$.
    
   \begin{obs} \label{obs:lawson-categories}
    The definition of an Ehresmann semigroupoid naturally leads to the notion of an \emph{Ehresmann category}, namely, a category $\mathcal{C}$ equipped with unary operations $+$ and $\ast$ such that $(\mathcal{C},+,\ast)$ is an Ehresmann semigroupoid. However, the term \emph{Ehresmann category} was already introduced by Lawson in \cite{lawson1991} with a different meaning. In the terminology adopted here, Lawson's Ehresmann categories are precisely the local biordered Ehresmann categories $(\mathcal{C},\leq_l,\leq_r)$ whose object set $\mathcal{C}_0$ is a meet-semilattice. To avoid ambiguity, throughout this paper we refer to Lawson's Ehresmann categories as \emph{biordered Ehresmann categories}, while the term \emph{Ehresmann category} is reserved for categories endowed with an Ehresmann semigroupoid structure.
\end{obs}
    
    Let $x,y \in \mathcal{C}$ be such that $e = D(x) \wedge R(y)$ is defined. Then $e \in \mathcal{C}_0$, $e \leq D(x)$ and $e \leq R(y)$ by \eqref{ec3}. From \eqref{ec1} and \eqref{ec2}, the restriction $x|e$ and the corestriction $e|y$ exist, and in this case we have $D(x|e) = e = R(e|y)$. Thus, the composition $(x|e) \circ (e|y)$ is defined. We introduce a new partial operation on $\mathcal{C}$, given by
        $$ x \otimes y = (x|e) \circ (e|y), \text{ whenever } e = D(x) \wedge R(y) \text{ is defined}. $$
    The partial operation $\otimes$ is called the \emph{pseudo-product} of $(\mathcal{C},\leq_l,\leq_r)$.

    \begin{lemma} \label{lema:pseudo-produto-1}
        The following conditions are equivalent.
        \begin{itemize}
            \item[\eqref{s1}] $x \otimes y$ and $y \otimes z$ are defined.
            \item[\eqref{s2}] $x \otimes y$ and $(x \otimes y) \otimes z$ are defined.
            \item[\eqref{s3}] $y \otimes z$ and $x \otimes (y \otimes z)$ are defined.
        \end{itemize}

        \begin{proof}
            Suppose that $x \otimes y$ is defined and note that $D(x \otimes y) = D(e|y)$. On the other hand, by \eqref{corestriction} we have $e|y \leq y$, and hence $D(e|y) \leq D(y)$ by \eqref{O1}. Since $\mathcal{C}_0$ is a local meet-semilattice, it follows that
                $$ D(x \otimes y) = D(e|y) = D(e|y) \wedge D(y). $$
            Using the associativity of $\wedge$ together with the equality above, we conclude that $D(x \otimes y) \wedge R(z)$ is defined if and only if $D(y) \wedge R(z)$ is defined. That is, conditions \eqref{s1} and \eqref{s2} are equivalent. Analogously, since $R(y \otimes z) = R(y|e) \leq R(y)$, we obtain that conditions \eqref{s1} and \eqref{s3} are equivalent.
        \end{proof}
    \end{lemma}

  We aim to show that $\mathcal{C}$ can be endowed with a semigroupoid structure, with composition given by $\otimes$. Verifying associativity, that is, $(x \otimes y) \otimes z = x \otimes (y \otimes z)$ whenever $x \otimes y$ and $y \otimes z$ are defined, turns out to be technically involved. Rather than carrying out this verification directly, we rely on a sequence of lemmas from \cite{lawson1991}. Their proofs adapt to the present setting with only minor modifications, using Lemma \ref{lema:pseudo-produto-1} to ensure that all expressions are well defined. Accordingly, we omit the details.

    The following lemma combines \cite[Lemmas 4, 10, and 17]{lawson1991}.

    \begin{lemma} \label{lema:restricao}
        Let $x \in \mathcal{C}$ and $e,f \in \mathcal{C}_0$.  Then the following statements hold.
        \begin{itemize}[align=parleft, labelsep=18px]
            \item[(a)] If $f \leq e \leq D(x)$, then $(x|e)|f = x|f$ and $x|f \leq_r x|e \leq_r x$.
            \item[(a')] If $f \leq e \leq R(x)$, then $f|(e|x) = f|x$ and $f|x \leq_l e|x \leq_l x$.
            \item[(b)] If $e \leq D(x)$, then $x|e = x \otimes e$.
            \item[(b')] If $e \leq R(x)$, then $e|x = e \otimes x$.
            \item[(c)] $e \otimes f$ is defined if and only if $e \wedge f$ is defined, and in this case $e \otimes f = e \wedge f = f \otimes e$.
        \end{itemize}
    \end{lemma}

    The following lemma combines \cite[Lemmas 12, 9, 11 and 19]{lawson1991}.

    \begin{lemma} \label{lema:pseudo-produto-2}
        Let $x,y,z \in \mathcal{C}$ and $e \in \mathcal{C}_0$. Then the following statements hold.
        \begin{itemize}
            \item[(a)] If $x \otimes y$ is defined, then $R(x \otimes y) = R(x \otimes R(y))$ and $D(x \otimes y) = D(D(x) \otimes y)$.
            \item[(b)] If $(x \otimes y) \otimes z$ is defined, then $(x \otimes y) \otimes z = (x|f) \otimes ((f|y) \otimes z)$, where $f = D(x) \wedge R(y)$.
            \item[(c)] If $(x \otimes e) \otimes z$ is defined, then $(x \otimes e) \otimes z = x \otimes (e \otimes z)$.
            \item[(d)] If $(e \otimes y) \otimes z$ is defined, then $(e \otimes y) \otimes z = e \otimes (y \otimes z)$.
        \end{itemize}
    \end{lemma}

    We can now conclude that the pseudo-product is associative.

    \begin{lemma} \label{lema:pseudo-produto-3}
      The set $\mathcal{C}$, endowed with the pseudo-product $\otimes$, is a semigroupoid.

        \begin{proof}
            Due to Lemma \ref{lema:pseudo-produto-1}, it remains to prove that, if $x \otimes y$ and $y \otimes z$ are defined, then $(x \otimes y) \otimes z = x \otimes (y \otimes z)$. Let $e = D(x) \wedge R(y)$. Note that $x \otimes R(y)$ is defined and that $x = x \circ D(x) = x \otimes D(x)$. Hence,
            \begin{align*}
                x \otimes R(y) &= (x \otimes D(x)) \otimes R(y) \\
                &= x \otimes (D(x) \otimes R(y)) & \ref{lema:pseudo-produto-2}(c) \\
                &= x \otimes e & \ref{lema:restricao}(c) \\
                &= x|e. & \ref{lema:restricao}(b)
            \end{align*}
            Analogously, we prove that $D(x) \otimes y = e|y$. Therefore,
            \begin{align*}
                [x \otimes y] \otimes z &= (x|e) \otimes [(e|y) \otimes z] & \ref{lema:pseudo-produto-2}(b) \\
                &= (x \otimes R(y)) \otimes [(D(x) \otimes y) \otimes z] \\
                &= (x \otimes R(y)) \otimes [D(x) \otimes (y \otimes z)] & \ref{lema:pseudo-produto-2}(d) \\
                &= [(x \otimes R(y)) \otimes D(x)] \otimes (y \otimes z) & \ref{lema:pseudo-produto-2}(c) \\
                &= [x \otimes (R(y) \otimes D(x)] \otimes (y \otimes z) & \ref{lema:pseudo-produto-2}(c) \\
                &= [x \otimes (D(x) \otimes R(y))] \otimes (y \otimes z) & \ref{lema:restricao}(c) \\
                &= [(x \otimes D(x)) \otimes R(y)] \otimes (y \otimes z) & \ref{lema:pseudo-produto-2}(d) \\
                &= (x \otimes D(x)) \otimes [R(y) \otimes (y \otimes z)] & \ref{lema:pseudo-produto-2}(d) \\
                &= (x \otimes D(x)) \otimes [(R(y) \otimes y) \otimes z] & \ref{lema:pseudo-produto-2}(d) \\
                &= (x|D(x)) \otimes [(R(y)|y) \otimes z] & \ref{lema:restricao}(b,b') \\
                &= x \otimes [y \otimes z],
            \end{align*}
            where the last equality follows from the uniqueness of restrictions and corestrictions.
        \end{proof}
    \end{lemma}

\subsection{The Correspondence} In this subsection, we describe the construction of the local biordered Ehresmann category $C(S)$ associated with an Ehresmann semigroupoid $(S,+,\ast)$ and the construction of the Ehresmann semigroupoid $S(\mathcal{C})$ associated with a local biordered Ehresmann category $(\mathcal{C},\leq_l,\leq_r)$. We then show that these constructions are mutually inverse, in the sense that $S(C(S)) = (S,+,\ast)$ and $C(S(\mathcal{C})) = (\mathcal{C},\leq_l,\leq_r)$.\\

    For each Ehresmann semigroupoid $(S,+,\ast)$, we associate the following structure.
    \begin{itemize}
        \item The sets $C(S) = S$ and $C(S)_0 = U$.
        \item Two functions $D,R \colon C(S) \to C(S)_0$, defined by $D(s) = s^\ast$ and $R(s) = s^+$.
        \item A partially defined binary operation on $C(S)$ given by $s \cdot t = st$ whenever $D(s) = R(t)$.
        \item Two partial orders $\leq_l$ and $\leq_r$ on $C(S)$, defined by
            $$ [s \leq_l t \iff s^+t \text{ is defined and } s^+t = s] \quad\text{and}\quad [s \leq_r t \iff ts^\ast \text{ is defined and } ts^\ast = s]. $$
    \end{itemize}
    By a slight abuse of notation, we denote the resulting data $(C(S),C(S)_0,D,R,\cdot,\leq_l,\leq_r)$ simply by $C(S)$.

    \begin{lemma}
        $C(S)$ is a local biordered Ehresmann category.

        \begin{proof}
            First, we prove that $(C(S),C(S)_0,D,R,\cdot)$ is a category. Notice that $s \cdot t$ is well defined. In fact, let $s,t \in C(S)$ be such that $s^\ast = D(s) = R(t) = t^+$. Then the composition $ss^\ast = st^+$ is defined by \eqref{re1}, and thus $s \cdot t = st$ is defined by Remark \ref{obs:le}. Now we verify conditions \eqref{C1}, \eqref{C2} and \eqref{C3}.

            Suppose that $s \cdot t$ is defined. That is, $s^\ast = t^+$. Then
            \begin{align*}
                D(s \cdot t) = (st)^\ast \overset{\eqref{re4}}{=} (s^\ast t)^\ast = (t^+t)^\ast \overset{\eqref{le1}}{=} t^\ast = D(t).
            \end{align*}
            Analogously, from \eqref{le4} and \eqref{re1}, we obtain that $R(s \cdot t) = R(s)$. This proves \eqref{C1}. Condition \eqref{C2} follows from the associativity of $S$. Indeed, if $s,t,r \in C(S)$ are such that $D(s) = R(t)$ and $D(t) = R(r)$, then
                $$ s \cdot (t \cdot r) = s(tr) = (st)r = (s \cdot t) \cdot r. $$
            Now, let $e \in C(S)_0 = U$. By Lemma \ref{lema:le}(a) and its dual, we have $D(e)=e^\ast=e$ and $R(e)=e^+=e$. Moreover, if $s^\ast=D(s)=e$, then $s \cdot e = ss^\ast = s$ by \eqref{re1}. Similarly, if $t^+ = R(t) = e$, then $e \cdot t = t^+t = t$ by \eqref{le1}. Hence, condition \eqref{C3} holds, and $C(S)$ is a category. We now proceed to show that $C(S)$ is a local biordered Ehresmann category.\\

            \noindent\eqref{ec1} Let $s,t \in C(S)$ be such that $s \leq_l t$. That is, $s^+t$ is defined and $s = s^+t$. Then
                $$ R(s) = s^+ = (s^+t)^+ \overset{\eqref{le3}}{=} s^+t^+ \overset{\ref{lema:le-order}(b)}{=} (s^+)^+ t^+ = R(s)^+ R(t), $$
            and
                $$ D(s) = s^\ast = (s^+t)^\ast \overset{\eqref{re1}}{=} ((s^+t)t^\ast)^\ast \overset{\eqref{re3}}{=} (s^+t)^\ast t^\ast = s^\ast t^\ast \overset{\eqref{lre}}{=} (s^\ast)^+ t^\ast = D(s)^+ D(t).  $$
            Therefore, $R(s) \leq_l R(t)$ and $D(s) \leq_l D(t)$. This proves \eqref{O1}. For \eqref{O2}, suppose that $s \leq_l t$, $u \leq_l v$ and that $s \cdot u$ and $t \cdot v$ are defined. Then $(su)^+(tv)$ is defined. In fact, from Lemma \ref{lema:le}(b) we obtain that $(su)^+ = (su)^+s^+$, and since $s \leq_l t$, the composition $s^+t$ is defined. From the associativity of $S$, we conclude that $(su)^+(tv) = (su)^+(s^+t)v$ is defined, and in this case
            \begin{align*}
                (su)^+(tv) &= (su)^+(s^+t)v & \ref{lema:le}(b) \\
                &= (su)^+ sv & (s \leq_l t) \\
                &= (su)^+ (ss^\ast) v & \eqref{re1} \\
                &= (su)^+ s(u^+ v) & (\exists s \cdot u) \\
                &= (su)^+ su & (u \leq_l v) \\
                &= su. & \eqref{le1}
            \end{align*}
            That is, $su \leq_l tv$. Now, we prove that $(\mathcal{C},\leq_l)$ has corestrictions. Let $e \in S(C)_0$ be such that $e \leq_l R(s)$. We claim that $e|s = es$. In fact, from $e \leq_l R(s)$, we obtain $e = es^+$. Therefore, by Remark \ref{obs:le}, $es$ is defined. Since $e = e^+$, it follows from \eqref{le3} that
                $$ (es)^+ = e^+s^+ = es^+ = e. $$
            That is, $R(es) = e$. Multiplying the previous equality by $s$ on the right side, we obtain $(es)^+s = es$. Thus, $es \leq_l s$. Suppose that $t \in C(S)$ is another element satisfying $s^+ = R(t) = e$ and $t \leq_l s$. Then $t = t^+s = es$. This shows that $e|s = es$ is the unique element satisfying $R(e|s) = e$ and $e|s \leq_l s$.\\

            \noindent\eqref{ec2} Is dual to \eqref{ec1}. The restriction $s|e$ is given by $se$, whenever $e \leq D(s)$.\\

            \noindent\eqref{ec3} We claim that $C(S)_0=U$ is a local meet-semilattice, with operation given by $e \wedge f = ef$, whenever $ef$ is defined in $S$. In fact, if $e,f \in U$ and $ef$ is defined, then $ef = e^+f^+ = (e^+f)^+ \in U$ by \eqref{le3}. Hence, the composition in $S$ restricts to a partially defined operation $(U \times U) \cap S^{(2)} \to U$. Moreover, by \eqref{le2}, whenever $ef$ is defined, so is $fe$, and in this case $ef=fe$. By Lemma \ref{lema:le}(a), for every $e \in U$ the product $ee$ is defined and satisfies $ee=e$. Therefore, $C(S)_0$ is a local meet-semilattice.\\

            \noindent\eqref{ec4} This is Lemma \ref{lema:lre-order}.  \\

            \noindent\eqref{ec5} Define an auxiliary relation on $C(S)$ by $s \leq_e t$ if and only if there are $u,v \in U$ such that $ut$ and $tv$ are defined and $utv = s$. We claim that $\leq_l \circ \leq_r = \leq_e = \leq_r \circ \leq_l$. In fact, suppose that there is $r \in C(S)$ such that $s \leq_l r \leq_r t$. Then
                $$ s = s^+r = s^+(tr^\ast), $$
            where $s^+,r^\ast \in U$. Therefore, $s \leq_e t$. Conversely, suppose that $s \leq_e t$ and let $u,v \in U$ be such that $s = utv$. Let $r = tv$. Since $u,v \in U$ and $s = utv = ur$, it follows from Lemma \ref{lema:le-order} and its dual that $r \leq_r t$ and $s \leq_l r$. This shows that $\leq_l \circ \leq_r = \leq_e$. The proof of $\leq_r \circ \leq_r = \leq_e$ is analogous.\\
            
            \noindent\eqref{ec6} Suppose that $s \leq_l t$ and that $D(s) \wedge e$ is defined. That is, $s = s^+t$ and $s^\ast e$ are defined. Since $s^\ast e \leq_l s^\ast = D(s)$, it follows from \eqref{ec2} that the restriction $s|(D(s) \wedge e)$ exists and is given by $s|(D(s) \wedge e) = ss^\ast e = se$. On the other hand, since $s = s^+t$, we obtain from the dual of Lemma \ref{lema:le}(b) and \eqref{re2} that $s^\ast = (s^+t)^\ast = (s^+t)^\ast t^\ast$. Now, it follows from $s^\ast e = (s^+t)^\ast t^\ast e$ that $t^\ast e = D(t) \wedge e$ is defined. Thus, the restriction $t|(D(t) \wedge e)$ exists and is given by $t|(D(t) \wedge e) = tt^\ast e = te$. Therefore
                $$ s|(D(s) \wedge e) = se = (s^+t)e = s^+(te) \leq_l te = t|(D(t) \wedge e). $$
            
            \noindent\eqref{ec7} Is dual to \eqref{ec6}. This proves that $(C(S),\leq_l,\leq_r)$ is a local biordered Ehresmann category.
        \end{proof}
    \end{lemma}

    For each local biordered Ehresmann category $(\mathcal{C},\leq_l,\leq_r)$, we associate the following structure.
    \begin{itemize}
        \item The sets $S(\mathcal{C}) = \mathcal{C}$ and $S(\mathcal{C})^{(2)} = \{ (x,y) \in \mathcal{C} \times \mathcal{C} \colon D(x) \wedge R(y) \text{ is defined} \}$.
        \item The pseudo-product $x \otimes y = (x|e) \circ (e|y)$, defined for every $(x,y) \in S(\mathcal{C})^{(2)}$.
        \item Two functions $+,\ast \colon S(\mathcal{C}) \to S(\mathcal{C})$, defined by $x^+ = R(x)$ and $x^\ast = D(x)$.
    \end{itemize}
    By a slight abuse of notation, we denote the resulting data $(S(\mathcal{C}), S(\mathcal{C})^{(2)},\otimes,+,\ast)$ simply by $S(\mathcal{C})$.

    \begin{lemma}
        $S(\mathcal{C})$ is an Ehresmann semigroupoid.

        \begin{proof}
            From Lemma and \ref{lema:pseudo-produto-3}, it follows that $\mathcal{S}(\mathcal{C})$ is a semigroupoid. Moreover, conditions \eqref{le2} and \eqref{re2} correspond to Lemma \ref{lema:restricao}(c), conditions \eqref{le4} and \eqref{re4} follow from Lemma \ref{lema:pseudo-produto-2}(a), and condition \eqref{lre} is precisely \eqref{C3}. To complete the proof, it remains to verify conditions \eqref{le1} and \eqref{le3}; the verification of \eqref{re1} and \eqref{re3} is dual.\\

            \noindent\eqref{le1} Let $x \in S(\mathcal{C})$. By \eqref{C3}, we have $D(s^+) = D(R(s)) = R(s)$, and hence $s^{+} \otimes s$ is defined. Moreover, by Lemma \ref{lema:restricao}(b'), we obtain $s^+ \otimes s = R(s)|s$. But, from \eqref{ec1}, $s$ must be the unique element satisfying $R(s) = R(s)$ and $s \leq_l s$. Therefore, $s^+ \otimes s = R(s)|s = s$.\\

            \noindent\eqref{le3} Suppose that $s^+ \otimes t^+$ is defined. Since condition \eqref{le1} is satisfied and the pseudo-product $\otimes$ is associative, it follows from Remark \ref{obs:le} that $s^+ \otimes t$ is defined. Furthermore,
                $$ (s^+ \otimes t)^+ \overset{\ref{lema:restricao}(b')}{=} R( (s^+ \wedge t^+)|t ) \overset{\eqref{ec1}}{=} s^+ \wedge t^+ \overset{\ref{lema:restricao}(c)}{=} s^+ \otimes t^+. $$
            This proves that $(S(\mathcal{C}),+,\ast)$ is an Ehresmann semigroupoid.
        \end{proof}
    \end{lemma}

    The following proposition provides a bijective correspondence between the class of Ehresmann semigroupoids and the class of local Ehresmann categories.

    \begin{theorem}\label{prop:ESN-objetos}
        With the above notation, we have:
        \begin{itemize}
            \item[(a)] $S(C(S)) = (S,+,\ast)$ as Ehresmann semigroupoids;
            \item[(b)] $C(S(\mathcal{C})) = (\mathcal{C},\leq_l,\leq_r)$ as local Ehresmann categories.
        \end{itemize}

        \begin{proof}
            (a) As sets we have $S(C(S)) = C(S)_1 = S$. Combining Remark \ref{obs:le} and its dual, we obtain that $st$ is defined if and only if $s^\ast t^+$ is defined. Therefore,
            \begin{align*}
                s \otimes t \text{ is defined} \iff D(s) \wedge R(t) = s^\ast t^+ \text{ is defined} \iff st \text{ is defined}.
            \end{align*}
            That is, $S(C(S))^{(2)} = S^{(2)}$. Suppose that $s \otimes t$ is defined and let $e = D(s) \wedge R(t) \in U$. Then
            \begin{align*}
                s \otimes t = (s|e) \circ (e|t) = seet \overset{\ref{lema:le}(a)}{=} set = ss^\ast t^+t \overset{\eqref{le1},\eqref{re1}}{=} st.
            \end{align*}
            Therefore, $S(C(S)) = S$ as semigroupoid. The left Ehresmann structure of $S(C(S))$ is given by $s^+ = R(s)$, while the range function of $C(S)$ is given by $R(s) = s^+$. Hence, $S(C(S))$ and $S$ have the same left Ehresmann structure. Analogously, we conclude that $S(C(S))$ and $S$ have the same right Ehresmann structure. Thus, $S(C(S)) = (S,+,\ast)$ as Ehresmann semigroupoids.\\

            \noindent(b) As sets we have $C(S(\mathcal{C})) = \mathcal{C}$ and $C(S(\mathcal{C}))_0 = U = \mathcal{C}_0$. The range function of $C(S(\mathcal{C}))$ is given by $R(x) = x^+$, while the left Ehresmann structure of $S(\mathcal{C})$ is given by $x^+ = R(x)$. Therefore, $C(S(\mathcal{C}))$ and $\mathcal{C}$ have the same range function $R$. Analogously, we conclude that $C(S(\mathcal{C}))$ and $\mathcal{C}$ have the same domain function $D$. Suppose that $D(x) = R(y)$. Then
            \begin{align*}
                x \cdot y = x \otimes y = (x|(D(x) \wedge R(y))) \circ ((D(x) \wedge R(y))|y) = (x|D(x)) \circ (R(y)|y) = x \circ y.
            \end{align*}
            Therefore, $C(S(\mathcal{C})) = \mathcal{C}$ as categories. Denote by $\prec_l$ and $\prec_r$ the partial orders of $C(S(\mathcal{C}))$ and by $\leq_l$ and $\leq_r$ the partial orders of $\mathcal{C}$. From Lemma \ref{lema:oc-order}, we obtain
                $$ x \prec_l y \iff x^+ \prec_l y^+ \text{ and } x^+|y = x, $$
            and
                $$ x \leq_l y \iff x^+ \leq_l y^+ \text{ and } x^+|y = x. $$
            From \eqref{ec4}, we obtain that the restriction of $\prec_l$ and of $\leq_l$ to $\mathcal{C}_0$ coincide with the partial order $\leq$ of $\mathcal{C}_0$. Hence, $x^+ \prec_l y^+$ if and only if $x^+ \leq_l y^+$. The corestriction on $C(S(\mathcal{C}))$ is given by $x^+|y = x^+ \otimes y$. On the other hand, it follows from Lemma \ref{lema:restricao}(b') that the corestriction on $\mathcal{C}$ satisfies $x^+|y = x^+ \otimes y$. Since both pseudo-products are given in $S(\mathcal{C})$, we conclude that $x \prec_l y$ if and only if $x \leq_l y$. Analogously, we have $x \prec_r y$ if and only if $x \leq_r y$. Thus, $C(S(\mathcal{C})) = (\mathcal{C},\leq_l,\leq_r)$ as local Ehresmann categories.
        \end{proof}
    \end{theorem}

\section{Particular Cases of the Correspondence} 
In this section, we apply the correspondence between Ehresmann semigroupoids and local Ehresmann categories established in Theorem \ref{prop:ESN-objetos} to derive several important special cases. We show that Theorem \ref{prop:ESN-objetos} extends the object part of the ESN Theorem for Ehresmann semigroups \cite[Theorem 4.24]{lawson1991}, and, consequently, for restriction semigroups \cite[Theorem 3.9]{hollings2010extending} and inverse semigroups \cite[Theorem 4.1.8]{lawson1998inverse}.  We also recover, as a parallel case, the correspondence for inverse semigroupoids \cite[Theorem 3.16]{dewolf2018ehresmann}. Moreover, we establish two new correspondences: two-sided restriction semigroupoids correspond to locally inductive categories, and Ehresmann categories correspond to complete local biordered Ehresmann categories. 

\subsection{Restriction Semigroupoids} Restriction semigroupoids were introduced in \cite{rsgpdexpansion}. In this subsection, we introduce the notion of locally inductive categories, which generalizes the inductive categories defined in \cite{hollings2010extending}. We then prove that $(S,+,\ast)$ is a restriction semigroupoid if and only if $C(S)$ is a locally inductive category.

    \begin{defi} \cite[Definition 2.5]{rsgpdexpansion}
        A \emph{left restriction semigroupoid} is a pair $(S,+)$, where $S$ is a semigroupoid and $+ \colon S \to S$ is a unary operation, satisfying the following properties:
        \begin{enumerate} \Not{lr}
            \item for all $s \in S$, $s^+s$ is defined and $s^+s = s$; \label{lr1}
            \item if $s^+t^+$ is defined, then $t^+s^+$ is defined and $s^+t^+ = t^+s^+$; \label{lr2}
            \item if $s^+t^+$ is defined, then $s^+t^+ = (s^+t)^+$; \label{lr3}
            \item if $st$ is defined, then $st^+ = (st)^+s$. \label{lr4}
        \end{enumerate}
        A \emph{right restriction semigroupoid} is a pair $(S,\ast)$, where $S$ is a semigroupoid and $\ast \colon S \to S$ is a unary operation, satisfying the following properties:
        \begin{enumerate} \Not{rr}
            \item for all $s \in S$, $ss^\ast$ is defined and $ss^\ast = s$; \label{rr1}
            \item if $s^\ast t^\ast$ is defined, then $t^\ast s^\ast$ is defined and $s^\ast t^\ast = t^\ast s^\ast$; \label{rr2}
            \item if $s^\ast t^\ast$ is defined, then $s^\ast t^\ast = (st^\ast)^\ast$; \label{rr3}
            \item if $st$ is defined, then $s^\ast t = t(st)^\ast$. \label{rr4}
        \end{enumerate}
        A \emph{two-sided restriction semigroupoid} is a triple $(S,+,\ast)$, where $(S,+)$ is a left restriction semigroupoid, $(S,\ast)$ is a right restriction semigroupoid and, for every $s \in S$, the following identities are satisfied:
        \begin{align*}
            s^+ = (s^+)^\ast \quad\text{and}\quad s^\ast = (s^\ast)^+. \tag{E} \label{lrr}
        \end{align*}
    \end{defi}

Throughout this paper, we work with two-sided restriction semigroupoids. For simplicity, we omit the term \emph{two-sided} and write \emph{restriction semigroupoid}.

    \begin{lemma}
        Every restriction semigroupoid is an Ehresmann semigroupoid.

        \begin{proof}
            We prove that left restriction semigroupoids are left Ehresmann semigroupoids. The proof for the right side is dual, and the two-sided version is a consequence of the latter.

            Let $(S,+)$ be a left restriction semigroupoid. We only need to prove that, if $st$ is defined, then $(st)^+ = (st^+)^+$. In fact, we have
            \begin{align*}
                (st)^+ \overset{\eqref{lr1}}{=} (s^+(st))^+ \overset{\eqref{lr3}}{=} s^+(st)^+ \overset{\eqref{lr2}}{=} (st)^+ s^+ \overset{\eqref{lr3}}{=} ((st)^+ s)^+ \overset{\eqref{lr4}}{=} (st^+)^+.
            \end{align*}
        \end{proof}
    \end{lemma}

    \begin{lemma} \label{lema:restriction-order}
        Let $(S,+,\ast)$ be a restriction semigroupoid. Then $\leq_l = \leq_r$.

        \begin{proof}
            Suppose that $s \leq_l t$. That is, $s^+t$ is defined and $s^+t = s$. Then
            \begin{align*}
                s = s^+ t \overset{\eqref{lrr}}{=} (s^+)^\ast t \overset{\eqref{rr4}}{=} t(s^+t)^\ast = ts^\ast.
            \end{align*}
            Therefore $s \leq_r t$. Analogously, if $s \leq_r t$, then $s \leq_l t$ by \eqref{lrr} and \eqref{lr4}.
        \end{proof}
    \end{lemma}

    \begin{defi}
        Let $\mathcal{C}$ be a category and $\leq$ be a partial order on $\mathcal{C}_1$. Then $(\mathcal{C},\leq)$ is called a \emph{locally inductive category} if it satisfies the following conditions:
        \begin{enumerate} \Not{ic}
            \item $(\mathcal{C},\leq)$ is an ordered category with restrictions and corestrictions; \label{ic1}
            \item $\mathcal{C}_0$ is a local meet-semilattice with operation $\wedge$ and partial order $\leq$. \label{ic2}
        \end{enumerate}
    \end{defi}

    \begin{prop} \label{prop:restriction}
        $(\mathcal{C},\leq)$ is a locally inductive category if and only if $(\mathcal{C},\leq,\leq)$ is a local biordered Ehresmann category.

        \begin{proof}
            Let $(\mathcal{C},\leq,\leq)$ be a local biordered Ehresmann category. Then \eqref{ic1} follows from conditions \eqref{ec1} and \eqref{ec2}, and \eqref{ic2} is precisely condition \eqref{ec3}. Therefore $(\mathcal{C},\leq)$ is a locally inductive category.
        
            Conversely, let $(\mathcal{C},\leq)$ be a locally inductive category. Then conditions \eqref{ec1} and \eqref{ec2} follows from \eqref{ic1}, condition \eqref{ec3} is \eqref{ic2}, and conditions \eqref{ec4} and \eqref{ec5} are trivially satisfied since $\leq_r = \leq_l = \leq$. We prove that condition \eqref{ec6} satisfied. The proof of \eqref{ec7} is dual to \eqref{ec6}.
            
            Assume that $x \leq y$ and that $D(x) \wedge e$ is defined. From $x \leq y$ and \eqref{ic1} we obtain that $D(x) \leq D(y)$. From \eqref{ic2}, $\mathcal{C}_0$ is a local meet-semilattice with partial order $\leq$, hence $D(y) \wedge e$ is defined and $D(x) \wedge e \leq D(y) \wedge e \leq D(y)$. Therefore, the restriction $(y|(D(y) \wedge e))|(D(x) \wedge e)$ exists and we have
                $$ (y|(D(y) \wedge e))|(D(x) \wedge e) \leq y|(D(y) \wedge e) \leq y, $$
            and
                $$ D((y|(D(y) \wedge e))|(D(x) \wedge e)) = D(x) \wedge e. $$
            From the uniqueness of the restriction, we obtain
                $$ y|(D(x) \wedge e) = (y|(D(y) \wedge e))|(D(x) \wedge e) \leq y|(D(y) \wedge e). $$
            On the other hand, since $D(x) \wedge e \leq D(x)$, the restriction $x|(D(x) \wedge e)$ is defined. But
                $$ x|(D(x) \wedge e) \leq x \leq y \quad\text{and}\quad D(x|(D(x) \wedge e)) = D(x) \wedge e. $$
            Therefore, it must be $x|(D(x) \wedge e) = y|(D(x) \wedge e) \leq y|(D(y) \wedge e)$. This proves \eqref{ec6}.
        \end{proof}
    \end{prop}

    \begin{lemma} \label{lema:restriction-corestriction}
        Let $(\mathcal{C},\leq)$ be a locally inductive category and $e \in \mathcal{C}_0$.
        \begin{itemize}
            \item[(a)] If $e \leq D(x)$, then $x|e = R(x|e)|x$.
            \item[(b)] If $e \leq R(x)$, then $e|x = x|D(e|x)$.
        \end{itemize}

        \begin{proof}
            (a) If $e \leq D(x)$, then the restriction $x|e$ is defined and $x|e \leq x$. Therefore, $R(x|e) \leq R(x)$ and the corestriction $R(x|e)|x$ is defined. By definition, $R(x|e)|x$ is the unique element satisfying $R(R(x|e)|x) = R(x|e)$ and $R(x|e)|x \leq x$. On the other hand, we have $R(x|e) = R(x|e)$ and $x|e \leq x$. Hence, it must be $x|e = R(x|e)|x$. (b) is dual to (a).
        \end{proof}
    \end{lemma}

    Since every restriction semigroupoid $(S,+,\ast)$ is an Ehresmann semigroupoid, it gives rise to a local biordered Ehresmann category $C(S)$. Conversely, since every locally inductive category $(\mathcal{C},\leq)$ can be regarded as a local biordered Ehresmann category, it determines an Ehresmann semigroupoid $S(\mathcal{C})$.

    Moreover, the assignments $(S,+,\ast) \mapsto C(S)$ and $(\mathcal{C},\leq_l,\leq_r) \mapsto S(\mathcal{C})$ are mutually inverse. Consequently, the result below allows us to restrict the correspondence established in Theorem \ref{prop:ESN-objetos} to a correspondence between restriction semigroupoids and locally inductive categories.

    \begin{theorem} \label{prop:ESN-objetos-restricao}
        $(S,+,\ast)$ is a restriction semigroupoid if and only if $C(S)$ is a locally inductive category.

        \begin{proof}
            Suppose that $(S,+,\ast)$ is a restriction semigroupoid. From Lemma \ref{lema:restriction-order} we have $\leq_r = \leq_l$. Since the partial orders of $C(S)$ are the partial orders of $(S,+,\ast)$, it follows from Proposition \ref{prop:restriction} that $C(S)$ is a locally inductive category.

            Conversely, suppose that $C(S)$ is a locally inductive category. We only need to prove that $S$ satisfies \eqref{lr4} and \eqref{rr4}. In fact, suppose that $s^\ast t$ is defined and let $e = s^\ast t^+ = D(s) \wedge R(t)$. Since $e \leq R(t)$, we have
            \begin{align*}
                t \otimes D(e|t) &= t \otimes D(e \otimes t) & \ref{lema:restricao}(b') \\
                &= t \otimes D(D(s) \otimes R(t) \otimes t) & \ref{lema:restricao}(c) \\
                &= t \otimes D(D(s) \otimes t) & \eqref{le1} \\
                &= t \otimes D(s \otimes t). & \ref{lema:pseudo-produto-2}(a)
            \end{align*}
            On the other hand, since $e|t \leq t$, it follows from \eqref{ic1} that $D(e|t) \leq D(t)$. Hence,
            \begin{align*}
                t \otimes D(e|t) &= t|(D(e|t)) & \ref{lema:restricao}(b) \\
                &= e|t & \ref{lema:restriction-corestriction}(b) \\
                &= e \otimes t & \ref{lema:restricao}(b') \\
                &= D(s) \otimes R(t) \otimes t \\
                &= D(s) \otimes t. & \eqref{le1}
            \end{align*}
            That is, $s^\ast t = D(s) \otimes t = t \otimes D(s \otimes t) = t(st)^\ast$. This proves \eqref{rr4}. The proof of \eqref{lr4} is analogous. Therefore, $(S,+,\ast)$ is a restrition semigroupoid.
        \end{proof}
    \end{theorem}

\subsection{Inverse Semigroupoids} Inverse semigroupoids have been extensively studied in the literature. In particular, the ESN Theorem for inverse semigroups \cite[Theorem 4.1.8]{lawson1998inverse} was generalized to inverse semigroupoids in \cite[Corollary 3.18]{dewolf2018ehresmann}, where it was shown that inverse semigroupoids correspond to locally inductive groupoids. In this subsection, we prove that locally inductive groupoids are precisely those locally inductive categories that are also groupoids. As a consequence of Theorem \ref{prop:ESN-objetos-restricao}, we recover the “object part” of the ESN Theorem for inverse semigroupoids.

    \begin{defi}
        A \emph{regular semigroupoid} is a semigroupoid $S$ such that, for each $s \in S$, there exists $t \in S$ with $st$ and $ts$ defined and satisfying $sts = s$ and $tst = t$. In this case, $t$ is called a \emph{pseudo-inverse} of $s$. An \emph{inverse semigroupoid} is a regular semigroupoid such that every element has a unique pseudo-inverse, denoted by $s^{-1}$.
    \end{defi}

    \begin{obs} \label{obs:inverse}
        (1) Let $S$ be a semigroupoid such that, for every $s \in S$, there exists $t \in S$ with $st$ and $ts$ defined and satisfying $sts = s$. Then $S$ is regular. A straightforward computation shows that, in this case, $r = tst$ is a pseudo-inverse of $s$.

        (2) By \cite[Theorem 2.25]{cordeiro2023etale}, a regular semigroupoid $S$ is an inverse semigroupoid if and only if the set of idempotent elements $E(S) = \{ e \in S \colon ee \text{ is defined and } ee = e \}$ is commutative, in the sense that whenever $e,f \in E(S)$ and $ef$ is defined, then $fe$ is also defined and $ef = fe$. 

        (3) Let $S$ be an inverse semigroupoid. Then $(s^{-1})^{-1} = s$ for every $s \in S$, and whenever $st$ is defined, we have $(st)^{-1} = t^{-1}s^{-1}$. Both identities follow from the uniqueness of the pseudo-inverse.

        (4) Every inverse semigroupoid can be regarded as a restriction semigroupoid by defining $s^+ = ss^{-1}$ and $s^\ast = s^{-1}s$. Therefore, by Theorem \ref{prop:ESN-objetos-restricao}, to each inverse semigroupoid $S$ we can associate a locally inductive category $C(S)$.
    \end{obs}

    We recall that a \emph{groupoid} is a category $\mathcal{G}$ such that every element is invertible, in the sense that for $x \in \mathcal{G}$, there exists a unique element $y \in \mathcal{G}$ such that $xy$ and $yx$ are defined and satisfy $xy = R(x)$ and $yx = D(x)$. In this case, we write $y = x^{-1}$.

   \begin{defi} \cite[Definition 3.4]{dewolf2018ehresmann}
        A \emph{locally inductive groupoid} is a pair $(\mathcal{G},\leq)$, where $\mathcal{G}$ is a groupoid, $\leq$ is a partial order on $\mathcal{G}$, and the following properties are satisfied:
        \begin{enumerate} \Not{ig}
            \item If $x \leq y$, then $x^{-1} \leq y^{-1}$. \label{ig1}
            \item If $x \leq y$, $u \leq v$ and $xu$ and $yv$ are defined, then $xu \leq yv$. \label{ig2}
            \item If $e \in \mathcal{G}_0$ and $e \leq D(x)$, then there exists a unique $x|e \in \mathcal{G}$ such that $x|e \leq x$ and $D(x|e) = e$. The element $x|e$ is called the \emph{restriction} of $x$ to $e$. \label{ig3}
            \item If $e \in \mathcal{G}_0$ and $e \leq R(x)$, then there exists a unique $e|x \in \mathcal{G}$ such that $e|x \leq x$ and $R(e|x) = e$. The element $e|x$ is called the \emph{corestriction} of $x$ to $e$. \label{ig4}
            \item $\mathcal{G}_0$ is a local meet-semilattice with operation $\wedge$ and partial order $\leq$. \label{ig5}
        \end{enumerate}
    \end{defi}

    The above definition of locally inductive groupoid generalizes the definition of inductive groupoid, used in the classical ESN Theorem for inverse semigroups. The only difference between the definitions is condition \eqref{ig5}, where for inductive groupoids the set $\mathcal{G}_0$ is assumed to be a meet-semilattice instead of a local one.

    \begin{lemma}\label{lema:loc-cat}
        Every locally inductive groupoid is a locally inductive category.

        \begin{proof}
            Conditions \eqref{ig2}, \eqref{ig3}, \eqref{ig4} and \eqref{ig5} are precisely \eqref{O2}, \eqref{restriction}, \eqref{corestriction} and \eqref{ic2}. We prove that \eqref{O1} is satisfied. Suppose that $x \leq y$. Then $x^{-1} \leq y^{-1}$ by \eqref{ig1}. Since $xx^{-1}$ and $yy^{-1}$ are defined, it follows from \eqref{ig2} that
                $$ R(x) = xx^{-1} \leq yy^{-1} = R(y). $$
            Analogously, since $x^{-1}x$ and $y^{-1}y$ are defined, we obtain that $D(x) \leq D(y)$.
        \end{proof}
    \end{lemma}

   Moreover, locally inductive groupoids are precisely the locally inductive categories that are also groupoids. This characterization is established in Lemma \ref{lema:inverse-2}. We first present an auxiliary lemma.

    \begin{lemma} \label{lema:inverse-1}
        $S$ is an inverse semigroupoid if and only if $C(S)$ is a locally inductive category and a groupoid.

        \begin{proof}
            Suppose that $S$ is an inverse semigroupoid. Then it is a restriction semigroupoid with $s^+ = ss^{-1}$ and $s^\ast = s^{-1}s$. Thus, $C(S)$ is a locally inductive category by Theorem \ref{prop:ESN-objetos-restricao}. Let $s \in C(S)$. We claim that the pseudo-inverse $s^{-1}$ of $s$ in $S$ is an inverse for $s$ in $C(S)$. In fact, we have
                $$ D(s^{-1}) = (s^{-1})^\ast = (s^{-1})^{-1} s^{-1} = ss^{-1} = R(s), $$
            and analogously $R(s^{-1}) = D(s)$. Therefore, $s^{-1} \circ s$ and $s \circ s^{-1}$ are defined, and in this case
                $$ s^{-1} \circ s = s^{-1}s = s^\ast = D(s) \quad\text{and}\quad s^{-1} \circ s = s^{-1}s = s^+ = R(s). $$
            This proves that $C(S)$ is a groupoid.

           Conversely, suppose that $C(S)$ is both a locally inductive category and a groupoid. We show that $S$ is a regular semigroupoid and that $E(S) \subseteq U$. Since $U$ is commutative by \eqref{rr2}, it then follows from Remark \ref{obs:inverse}(2) that $S$ is an inverse semigroupoid. Let $s \in S$. Since $D(s) = R(s^{-1})$ we have that $s \otimes s^{-1}$ is defined, and
            \begin{align*}
                s \otimes s^{-1} &= [s|(D(s) \wedge R(s^{-1}))] \circ [(D(s) \wedge R(s)^{-1})|s^{-1}] \\
                &= [s|D(s)] \circ [R(s)^{-1}|s^{-1}] \\
                &= s \circ s^{-1} \\
                &= R(s).
            \end{align*}
            Furthermore, since $D(R(s)) = R(s)$, it follows that $R(s) \otimes s$ is defined. From Lemma \ref{lema:restricao}(b'), we conclude that $R(s) \otimes s = R(s)|s = s$. Therefore $s \otimes s^{-1} \otimes s = R(s) \otimes s = s$. By Remark \ref{obs:inverse}(1), we obtain that $S$ is a regular semigroupoid. On the other hand, suppose that $e \in E(S)$ and let $f = D(e) \wedge R(e)$. From $e = e \otimes e$, we have
                $$ D(e) = D((e|f) \circ (f|e)) = D(f|e). $$
            Since $f|e \leq e$, it follows from the uniqueness of the restriction that $f|e = e|D(e) = e$. Analogously, we have $R(e) = R(e|f)$ and $e|f \leq e$. Therefore $e|f = R(e)|e = e$. But then
                $$ e^+ = R(e) = e \circ e^{-1} = (e|f) \circ (f|e) \circ e^{-1} = e \circ e \circ e^{-1} = e \circ R(e) \overset{\eqref{C3}}{=} e. $$
            This proves $E(S) \subseteq U$. Since $C(S)$ is a locally inductive category, $S$ is a restriction semigroupoid. Therefore $U$, hence $E(S)$, is commutative by \eqref{lr2}. By Remark \ref{obs:inverse}(2), we conclude that $S$ is an inverse semigroupoid. 
        \end{proof}
    \end{lemma}

    \begin{lemma} \label{lema:inverse-2}
        If $C(S)$ is a locally inductive category and a groupoid, then it is a locally inductive groupoid.

        \begin{proof}
            It remains to prove that \eqref{ig1} is satisfied. Since $C(S)$ is a locally inductive category and a groupoid, it follows from Lemma \ref{lema:inverse-1} that $S$ is an inverse semigroupoid. Therefore,
            \begin{align*}
                s \leq t &\iff s^+t \text{ is defined and } s^+t = s \iff s^{-1}t \text{ is defined and } ss^{-1}t = s, \\
                &\iff ts^\ast \text{ is defined and } ts^\ast = s \iff ts^{-1} \text{ is defined and } ts^{-1}s = s.
            \end{align*}
            From Remark \ref{obs:inverse}(3), we obtain $s^{-1} = (ss^{-1}t)^{-1} = t^{-1}(s^{-1})^{-1} s^{-1}$. Thus, $s^{-1} \leq t^{-1}$.
        \end{proof}
    \end{lemma}

    Combining Theorem \ref{prop:ESN-objetos} with Lemmas \ref{lema:loc-cat}, \ref{lema:inverse-1} and \ref{lema:inverse-2}, we obtain the correspondence between inverse semigroupoids and locally inductive groupoids, as in \cite[Corollary 3.18]{dewolf2018ehresmann}.

    \begin{theorem} \label{prop:ESN-inverse}
        $S$ is an inverse semigroupoid if and only if $C(S)$ is a locally inductive groupoid.
    \end{theorem}

\subsection{Semigroups and Categories} In this subsection, we consider the case in which the Ehresmann semigroupoid $S$ is either a semigroup or a category. We show that $S$ is a semigroup (respectively, a category) if and only if the local meet-semilattice $U$ is a semigroup (respectively, a category). Throughout this subsection, $(S,+,\ast)$ denotes an Ehresmann semigroupoid with set of projections $U$.\\ \label{ssec:3.3}

    Recall that a semigroupoid $S$ is a semigroup if and only if $S^{(2)} = S \times S$. In particular, if a local meet-semilattice $U$ is a semigroup, then it is a meet-semilattice. By Remark \ref{obs:lawson-categories}, local Ehresmann categories whose object set $\mathcal{C}_0$ is a meet-semilattice are called biordered Ehresmann categories.

    \begin{prop} \label{lema:ESN-semigrupo}
        The following conditions are equivalent:
        \begin{itemize}
            \item[(a)] $S$ is a semigroup.
            \item[(b)] $U$ is a meet-semilattice.
            \item[(c)] $C(S)$ is a biordered Ehresmann category.
        \end{itemize}

        \begin{proof}
            (b) is equivalent to (c) by definition. We prove that (a) is equivalent to (b). In fact, if $S$ is a semigroup, then $S^{(2)} = S \times S$. Hence,
                $$ U^{(2)} = (U \times U) \cap S^{(2)} = (U \times U) \cap (S \times S) = U \times U. $$
            Therefore, $U$ is a semigroup, hence a meet-semilattice. Conversely, suppose that $U$ is a meet-semilattice and let $s,t \in S$. Since $C(S)_0 = U$ is a meet-semilattice, it follows that $D(s) \wedge R(t)$ is defined. But then $s \otimes t$ is defined. Thus, it follows from Theorem \ref{prop:ESN-objetos} that $st = s \otimes t$ is defined. That is, $S^{(2)} = S \times S$ and $S$ is a semigroup.
        \end{proof}
    \end{prop}

    \begin{obs}
        In the context of the ESN-type theorem for one-sided restriction semigroupoids \cite[Lemma 7.11]{lrspgesn}, it was shown that a left restriction semigroupoid $(S,+)$ is a semigroup if and only if the set of projections $S^{+}$ is a meet-semilattice and $C(S)$ is non-degenerate, in the sense that for every $x \in C(S)$ there exists $e \in C(S)_0$ such that $x|e$ is defined. This additional condition is automatically satisfied for two-sided Ehresmann semigroupoids, since $x|x^{\ast} = x \otimes x^{\ast}$ is defined for all $x \in C(S)$. Hence, the simplicity of Proposition \ref{lema:ESN-semigrupo} is a consequence of the structural features of two-sided Ehresmann semigroupoids.
    \end{obs}

   Proposition \ref{lema:ESN-semigrupo} establishes the correspondence between Ehresmann semigroups and biordered Ehresmann categories, as originally formulated by Lawson in \cite[Theorem 4.24]{lawson1991}, which in turn generalizes the corresponding results for inverse semigroups \cite[Theorem 4.1.8]{lawson1998inverse} and for restriction semigroups \cite[Theorem 3.9]{hollings2010extending}. \\

    Motivated by the construction for inverse categories in \cite{dewolf2018ehresmann}, we now introduce the following terminology. Recall from Proposition \ref{prop:localmeet} that a local meet-semilattice $U$ is a disjoint union of meet-semilattices, which we refer to as its components. We call $U$ a \emph{locally complete meet-semilattice} if each of its components has a greatest element. Equivalently, the partial meet operation makes $U$ into a category, where the greatest element of each component is its identity, and, if $e$ belongs to a component with greatest element $\widehat{e}$, then $D(e)=R(e)=\widehat{e}$. In particular, this category structure is uniquely determined by the local meet-semilattice structure and is not additional data.

    A triple $(\mathcal{C},\leq_l,\leq_r)$ is called a \emph{complete local biordered Ehresmann category} if it is a local biordered Ehresmann category and $\mathcal{C}_0$ is a locally complete meet-semilattice.

    \begin{prop} \label{lema:ESN-cat}
        The following conditions are equivalent:
        \begin{itemize}
            \item[(a)] $S$ is a category.
            \item[(b)] $U$ is a locally complete meet-semilattice.
            \item[(c)] $C(S)$ is a complete  local biordered Ehresmann category.
        \end{itemize}

        \begin{proof}
           Once again, conditions (b) and (c) are equivalent by definition. We now prove that (a) is equivalent to (b).
            
            Suppose that $S$ has a category structure $D,R \colon S \to S_0$. Then, for every $e \in U \subseteq S$, there are identities $D(e),R(e) \in S_0 \subseteq S$ such that $eD(e)$ and $R(e)e$ are defined. Since $ee$ is defined, it follows that $D(e) = R(e)$. From \eqref{rr1} and \eqref{C3}, we obtain that $D(e) = D(e)D(e)^\ast = D(e)^\ast$. That is, $D(e) = R(e) \in U$. Therefore, the category structure on $S$ restricts to a category structure $D,R \colon U \to S_0 \subseteq U$ on $U$. This shows that $U$ is a locally complete meet-semilattice.

            Conversely, suppose that $U$ is a locally complete meet-semilattice with category structure $d,r \colon U \to U_0$. Define $D,R \colon S \to U_0$ by $D(s) = d(s^\ast)$ and $R(s) = r(s^+)$, for all $s \in S$. Then
                $$ s \otimes t \text{ is defined} \iff s^\ast \wedge t^+ \text{ is defined} \iff d(s^\ast) = r(t^+) \iff D(s) = R(t). $$
            By Theorem \ref{prop:ESN-objetos}, $s \otimes t$ is defined if and only if $st$ is defined. Thus, $st$ is defined if and only if $D(s) = R(t)$. Furthermore, if $D(s) = R(t)$, then $(st)^+ = s^+(st)^+$ by Lemma \ref{lema:le}(b). Hence,
                $$ R(st) = r((st)^+) = r(s^+(st)^+) = r(s^+) = R(s). $$
            Analogously, we conclude that $D(st) = D(t)$. Therefore, $(S,U_0,D,R,\otimes)$ satisfies \eqref{C1}. Condition \eqref{C2} follows from the associativity of the pseudo-product $\otimes$. To prove \eqref{C3}, let $e \in U_0 \subseteq S$. Then $e^+ = e = e^\ast$ by Lemma \ref{lema:le}(a) and its dual. Hence,
                $$ D(e) = d(e^\ast) = d(e) = e = r(e) = r(e^+) = R(e). $$
            Suppose that $d(s^\ast) = D(s) = e$. From Lemma \ref{lema:restricao}(b), we obtain that $s \otimes e = s|(s^\ast \wedge e)$. Hence
                $$ se  = s \otimes e = s|(s^\ast \wedge e) = s|(s^\ast \wedge d(s^\ast)) = s|s^\ast = s. $$
            Analogously, if $R(t) = e$, then $et = t$. This proves that $(S,U_0,D,R,\otimes)$ is a category.
        \end{proof}
    \end{prop}

    \begin{obs}
        Proposition \ref{lema:ESN-cat} shows that $(S,+,\ast)$ admits a category structure if and only if $U$ admits a category structure, and that, in this case, the structure in $S$ is an extension of the structure on $U$. In the context of the ESN-type theorem for one-sided restriction semigroupoids \cite[Lemma 7.9]{lrspgesn}, it was shown that $(S,+)$ admits a category structure if and only if $U$ admits a category structure and $C(S)$ is both non-degenerate and unitary, in the sense that for each identity $e \in U_0$, whenever the restriction $x|e$ is defined, one has $x|e = x$. Therefore, the fact that the category structure is always preserved by the correspondence is a distinctive feature of the two-sided version of the construction.
    \end{obs} 

    We conclude the section by summarizing the particular cases of the object part of the ESN-type theorems in two diagrams. The first diagram relates the subclasses of Ehresmann semigroupoids we have discussed in this section, Ehresmann semigroupoids being the most general class, and inverse monoids (inverse semigroupoids that are semigroups and categories) being the least general.

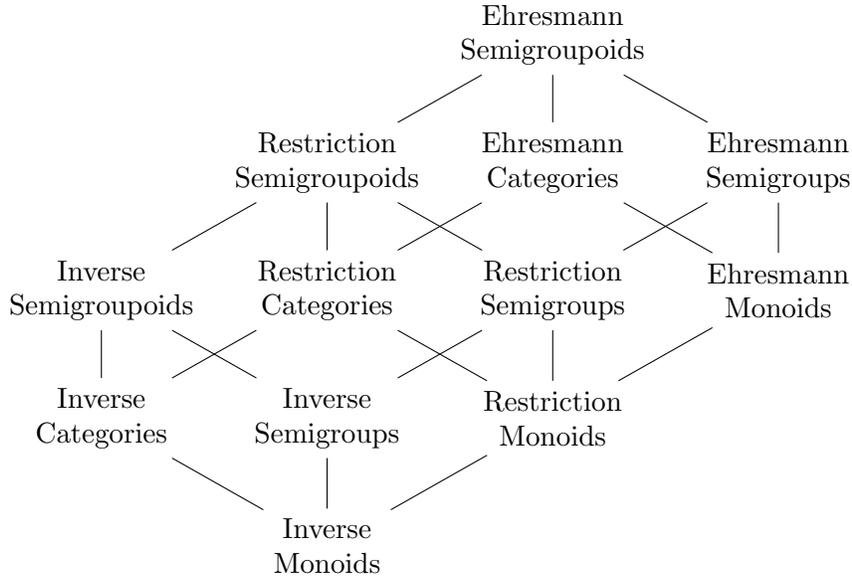
    \begin{center}
        \begin{tikzpicture}[yscale=0.85]
            \tikzstyle{every path}=[draw];
            \tikzstyle{every node}=[align=center];

            \node (Espd) at (0,0) {Ehresmann\\ Semigroupoids};
            \node (Ecat) at (0,-2) {Ehresmann\\ Categories};
            \node (Rspd) at (-3,-2) {Restriction\\ Semigroupoids};
            \node (Rcat) at (-3,-4) {Restriction\\ Categories};
            \node (Ispd) at (-6,-4) {Inverse\\ Semigroupoids};
            \node (Icat) at (-6,-6) {Inverse\\ Categories};
            \node (Esgp) at (3,-2) {Ehresmann\\ Semigroups};
            \node (Emon) at (3,-4) {Ehresmann\\ Monoids};
            \node (Rsgp) at (0,-4) {Restriction\\ Semigroups};
            \node (Rmon) at (0,-6) {Restriction\\ Monoids};
            \node (Isgp) at (-3,-6) {Inverse\\ Semigroups};
            \node (Imon) at (-3,-8) {Inverse\\ Monoids};

            \path (Espd) to (Rspd); \path (Espd) to (Ecat); \path (Espd) to (Esgp);
            \path (Rspd) to (Ispd); \path (Rspd) to (Rcat); \path (Rspd) to (Rsgp);
            \path (Ecat) to (Rcat); \path (Ecat) to (Emon);
            \path (Esgp) to (Rsgp); \path (Esgp) to (Emon);
            \path (Ispd) to (Icat); \path (Ispd) to (Isgp);
            \path (Rcat) to (Icat); \path (Rcat) to (Rmon);
            \path (Rsgp) to (Isgp); \path (Rsgp) to (Rmon);
            \path (Emon) to (Rmon);
            \path (Icat) to (Imon); \path (Isgp) to (Imon); \path (Rmon) to (Imon);
        \end{tikzpicture}
    \end{center}

    The second diagram relates the subclasses of local biordered Ehresmann categories that correspond to each subclass of Ehresmann semigroupoids through the correspondence on Theorem \ref{prop:ESN-objetos}. In this diagram, the prefix ``L.'' means ``local/locally'', the absence of this prefix indicates that $\mathcal{C}_0$ is a meet-semilattice. The prefix ``C.'' means ``complete'', indicating that $\mathcal{C}_0$ has a category structure. The prefix ``B.'' means ``biordered'', used to distinguish Lawson's Ehresmann categories from the Ehresmann categories in the previous diagram.

    
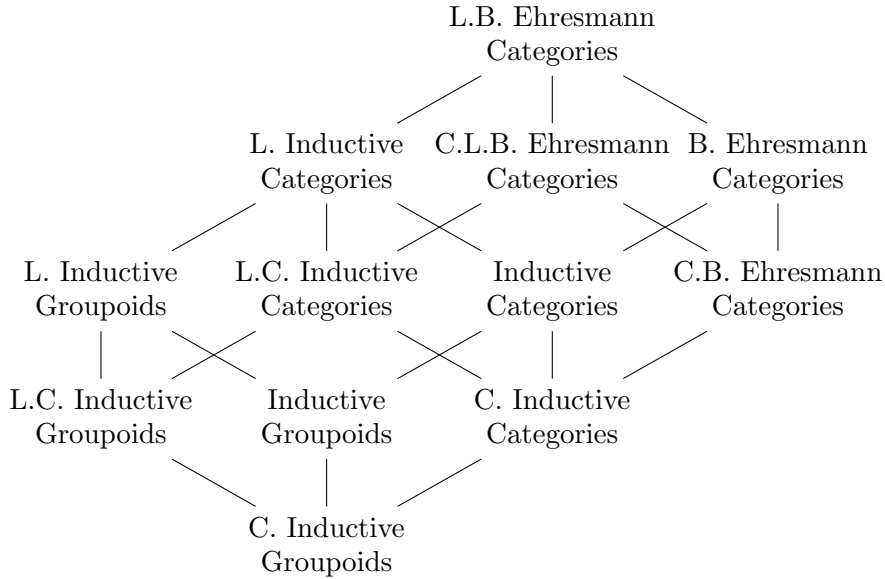
\begin{center}
        \begin{tikzpicture}[yscale=0.85]
            \tikzstyle{every path}=[draw];
            \tikzstyle{every node}=[align=center];

            \node (Espd) at (0,0) {L.B. Ehresmann\\ Categories};
            \node (Ecat) at (0,-2) {C.L.B. Ehresmann\\ Categories};

            \node (Rspd) at (-3,-2) {L. Inductive \\ Categories};
            \node (Rcat) at (-3,-4) {L.C. Inductive\\ Categories};

            \node (Ispd) at (-6,-4) {L. Inductive\\ Groupoids};
            \node (Icat) at (-6,-6) {L.C. Inductive\\ Groupoids};

            \node (Esgp) at (3,-2) {B. Ehresmann\\ Categories};
            \node (Emon) at (3,-4) {C.B. Ehresmann\\ Categories};

            \node (Rsgp) at (0,-4) {Inductive\\ Categories};
            \node (Rmon) at (0,-6) {C. Inductive\\ Categories};

            \node (Isgp) at (-3,-6) {Inductive\\ Groupoids};
            \node (Imon) at (-3,-8) {C. Inductive\\ Groupoids};

            \path (Espd) to (Rspd); \path (Espd) to (Ecat); \path (Espd) to (Esgp);
            \path (Rspd) to (Ispd); \path (Rspd) to (Rcat); \path (Rspd) to (Rsgp);
            \path (Ecat) to (Rcat); \path (Ecat) to (Emon);
            \path (Esgp) to (Rsgp); \path (Esgp) to (Emon);
            \path (Ispd) to (Icat); \path (Ispd) to (Isgp);
            \path (Rcat) to (Icat); \path (Rcat) to (Rmon);
            \path (Rsgp) to (Isgp); \path (Rsgp) to (Rmon);
            \path (Emon) to (Rmon);
            \path (Icat) to (Imon); \path (Isgp) to (Imon); \path (Rmon) to (Imon);
        \end{tikzpicture}
    \end{center}

\section{The ESN-type Theorems} In this section, we extend the correspondence established in Theorem \ref{prop:ESN-objetos} to an isomorphism between the category of Ehresmann semigroupoids with $(2,1,1)$-morphisms and the category of local biordered Ehresmann categories with inductive functors. By restricting to restriction semigroupoids, we obtain further ESN-type correspondences by weakening the notion of morphism in two distinct directions, which leads to two additional category isomorphisms. More precisely, we show that the category of restriction semigroupoids with $\vee$-premorphisms is isomorphic to the category of locally inductive categories with ordered functors, and that the category of restriction semigroupoids with $\wedge$-premorphisms is isomorphic to the category of locally inductive categories with inductive prefunctors.

\subsection{(2,1,1)-Morphisms and Inductive Functors} Throughout this subsection, $(S,+,\ast)$ and $(T,+,\ast)$ denote Ehresmann semigroupoids with sets of projections $U$ and $V$, respectively, while $(\mathcal{C},\leq_l,\leq_r)$ and $(\mathcal{D},\leq_l,\leq_r)$ denote local Ehresmann categories. We write $C(S)$ and $C(T)$ for the local Ehresmann categories corresponding to $S$ and $T$ under Theorem \ref{prop:ESN-objetos}.

    \begin{defi}
        A function $\varphi \colon S \to T$ between Ehresmann semigroupoids is called a \emph{(2,1,1)-morphism} if it satisfies the following conditions:
        \begin{enumerate} \Not{m}
            \item if $st$ is defined, then $\varphi(s)\varphi(t)$ is defined and $\varphi(s)\varphi(t) = \varphi(st)$; \label{m1}
            \item $\varphi(s^+) = \varphi(s)^+$ and $\varphi(s^\ast) = \varphi(s)^\ast$, for every $s \in S$. \label{m2}
        \end{enumerate}
    \end{defi}

    The term “(2,1,1)-morphism” refers to a function that preserves one partial binary operation and two unary operations, namely, the composition $\cdot \colon S^{(2)} \subseteq S \times S \to S$ and the left and right restriction operations $+ \colon S \to S$ and $\ast \colon S \to S$.

    Clearly, the composition of $(2,1,1)$-morphisms and the identity map $\mathrm{id}_S \colon S \to S$ are $(2,1,1)$-morphisms. Hence, Ehresmann semigroupoids together with $(2,1,1)$-morphisms form a category, which we denote by $ESGPD_m$.

    We recall that a \emph{full subcategory} of a category $\mathcal{C}$ is a category $\mathcal{C}'$ such that $\mathcal{C}'_0 \subseteq \mathcal{C}_0$ and, for every $X,Y \in \mathcal{C}'_0$, any morphism $x \in \mathcal{C}$ with $D(x) = X$ and $R(x) = Y$ belongs to $\mathcal{C}'$. In other words, $\mathcal{C}'$ contains all morphisms of $\mathcal{C}$ between its objects.

    We denote by $rSGPD_m$ and $iSGPD_m$ the full subcategories of $ESGPD_m$ whose objects are restriction semigroupoids and inverse semigroupoids, respectively.

    \begin{defi}
        A function $\varphi \colon \mathcal{C} \to \mathcal{D}$ between local Ehresmann categories is called an \emph{inductive functor} if it satisfies the following conditions:
        \begin{enumerate} \Not{if}
            \item $\varphi(D(x)) = D(\varphi(x))$ and $\varphi(R(x)) = R(\varphi(x))$, for every $x \in \mathcal{C}$; \label{if1}
            \item if $D(x) = R(y)$, then $\varphi(x) \circ \varphi(y) = \varphi(x \circ y)$; \label{if2}
            \item if $x \leq_l y$, then $\varphi(x) \leq_l \varphi(y)$; \label{if3}
            \item if $x \leq_r y$, then $\varphi(x) \leq_r \varphi(y)$; \label{if4}
            \item if $e \wedge f$ is defined, then $\varphi(e) \wedge \varphi(f)$ is defined and $\varphi(e \wedge f) = \varphi(e) \wedge \varphi(f)$. \label{if5}
        \end{enumerate}
    \end{defi}

    \begin{lemma} \label{lema:ESN-morfismos}
        Let $\varphi \colon S \to T$ be a function, and regard $\varphi \colon C(S) \to C(T)$ as the same underlying map. Then $\varphi \colon S \to T$ is a $(2,1,1)$-morphism if and only if $\varphi \colon C(S) \to C(T)$ is an inductive functor.

        \begin{proof}
            Suppose that $\varphi \colon S \to T$ is a (2,1,1)-morphism. We prove that $\varphi$ is an inductive functor.\\

            \noindent\eqref{if1} Is precisely \eqref{m2}.\\
            
            \noindent\eqref{if2} If $s \circ t$ is defined, then $st$ is defined and $s \circ t = st$. Hence, by \eqref{m1}, we have $\varphi(s \circ t) = \varphi(st) = \varphi(s)\varphi(t)$. From \eqref{if1}, we obtain $D(\varphi(s)) = \varphi(D(s)) = \varphi(R(t)) = R(\varphi(t))$. Therefore, $\varphi(s) \circ \varphi(t)$ is defined, and $\varphi(s) \circ \varphi(t) = \varphi(s)\varphi(t) = \varphi(s \circ t)$.\\
            
            \noindent\eqref{if3} Suppose that $s \leq_l t$. Then $s^+t$ is defined and $s^+t = s$. Hence,
                $$ \varphi(s) = \varphi(s^+t) \overset{\eqref{m1}}{=} \varphi(s^+)\varphi(t) \overset{\eqref{m2}}{=} \varphi(s)^+\varphi(t). $$
            That is, $\varphi(s) \leq_l \varphi(t)$.\\
            
            \noindent\eqref{if4} Is analogous to \eqref{if3}. \\
            
            \noindent\eqref{if5} Let $e,f \in C(S)_0$ be such that $e \wedge f$ is defined. That is, $ef$ is defined. By \eqref{m1}, it follows that $\varphi(e)\varphi(f)$ is defined and that $\varphi(e)\varphi(f) = \varphi(ef)$. On the other hand, since $e = e^{+}$ and $f = f^{+}$, we have $\varphi(e) = \varphi(e^+) = \varphi(e)^+ \in C(T)_0$ by \eqref{m2}, and similarly $\varphi(f) \in C(T)_0$. Hence, $\varphi(e) \wedge \varphi(f)$ is defined, and
                $$ \varphi(e) \wedge \varphi(f) = \varphi(e)\varphi(f) = \varphi(ef) = \varphi(e \wedge f). $$
            This shows that $\varphi$ is an inductive functor. 
            
            Conversely, suppose that $\varphi \colon C(S) \to C(T)$ is an inductive functor. We prove that $\varphi$ is a (2,1,1)-morphism.\\

            \noindent\eqref{m1} Suppose that $st$ is defined and set $e = s^\ast t^+ = D(s) \wedge R(t)$. We claim that $\varphi(s|e) = \varphi(s)|\varphi(e)$. Indeed, by \eqref{if1} we have $D(\varphi(s|e)) = \varphi(D(s|e)) = \varphi(e)$ by \eqref{if1}. Moreover, since $s|e \leq_r s$, it follows from \eqref{if4} that $\varphi(s|e) \leq_r \varphi(s)$. By the uniqueness of the restriction, we conclude that $\varphi(s|e) = \varphi(s)|\varphi(e)$. Analogously, we obtain $\varphi(e|t) = \varphi(e)|\varphi(t)$. On the other hand, by \eqref{if5} and \eqref{if1}, we have $\varphi(e) = \varphi(D(s) \wedge R(t)) = D(\varphi(s)) \wedge R(\varphi(t))$. Therefore, $\varphi(s) \otimes \varphi(t)$ is defined and
            \begin{align*}
                \varphi(s) \otimes \varphi(t) &= (\varphi(s)|\varphi(e)) \circ (\varphi(e)|\varphi(t)) \\
                &= \varphi(s|e) \circ \varphi(e|t) \\
                &= \varphi((s|e) \circ (e|t)) & \eqref{if2} \\
                &= \varphi(s \otimes t).
            \end{align*}
            Hence, $\varphi(s)\varphi(t) = \varphi(st)$.\\

            \noindent\eqref{m2} Is precisely \eqref{if1}. Thus, $\varphi$ is a (2,1,1)-morphism.
        \end{proof}
    \end{lemma}

\begin{lemma}\label{coro:composition-ifunctors}
The composition of inductive functors is an inductive functor.
\begin{proof}
Let $\varphi \colon \mathcal{C} \to \mathcal{D}$ and $\psi \colon \mathcal{D} \to \mathcal{E}$ be inductive functors. By Theorem \ref{prop:ESN-objetos}, we may assume that $\mathcal{C} = C(S)$, $\mathcal{D} = C(T)$, and $\mathcal{E} = C(R)$, where $S = S(\mathcal{C})$, $T = S(\mathcal{D})$, and $R = S(\mathcal{E})$ are Ehresmann semigroupoids.  

Since $\varphi$ and $\psi$ are inductive functors, it follows from Lemma \ref{lema:ESN-morfismos} that $\varphi \colon S \to T$ and $\psi \colon T \to R$ are $(2,1,1)$-morphisms. Therefore, their composition $\psi \circ \varphi \colon S \to R$ is also a $(2,1,1)$-morphism. Applying Lemma \ref{lema:ESN-morfismos} once again, we conclude that $\psi \circ \varphi \colon C(S) \to C(R)$ is an inductive functor.
\end{proof}
\end{lemma}

    From Lemma \ref{coro:composition-ifunctors}, we conclude that local Ehresmann categories together with inductive functors form a category, which we denote by $lECAT_{if}$. Furthermore, we denote by $liCAT_{if}$ and $liGPD_{if}$ the full subcategories of $lECAT_{if}$ whose objects are locally inductive categories and locally inductive groupoids, respectively.
    
    Now, we prove the ESN Theorem for Ehresmann semigroupoids.

    \begin{theorem} \label{teo:ESN-morphism}
        The categories $ESGPD_m$ and $lECAT_{if}$ are isomorphic.

        \begin{proof}
            We define a functor $C \colon ESGPD_m \to lECAT_{if}$ by assigning $S \mapsto C(S)$ for each Ehresmann semigroupoid $S$, and $\varphi \mapsto C\varphi := \varphi$ for each (2,1,1)-morphism. Notice that $C$ is a functor. Indeed, it follows by Theorem \ref{prop:ESN-objetos} that $C(S)$ is a local biordered Ehresmann category, and by Lemma \ref{lema:ESN-morfismos} that $C\varphi$ is an inductive functor. Moreover, if $\varphi \colon S \to T$ is a (2,1,1)-morphism, then
                $$ D(C\varphi) = C(S) = C(D(\varphi)) \quad\text{and}\quad R(C\varphi) = C(T) = C(R(\varphi)). $$
            If $\psi \colon T \to R$ is another (2,1,1)-morphism, then
                $$ C(\psi \circ \varphi) = \psi \circ \varphi = C\psi \circ C\varphi. $$
            Finally, we have $C1_S = 1_S = 1_{C(S)}$. This shows that $C \colon ESGPD_m \to lECAT_{if}$ is a functor. 
            
            Analogously, we define a functor $S \colon lECAT_{if} \to ESGPD_m$by assigning $\mathcal{C} \mapsto S(\mathcal{C})$ for each local biordered Ehresmann category $\mathcal{C}$, and $\varphi \mapsto S\varphi := \varphi$ for each inductive functor. From Theorem \ref{prop:ESN-objetos} and the fact that both $C$ and $S$ act as the identity on morphisms, it follows that $C \circ S = id_{ESGPD_m}$ and $S \circ C = id_{lECAT_{if}}$. Hence, $C$ and $S$ define isomorphisms between the categories $ESGPD_m$ and $lECAT_{if}$.
        \end{proof}
    \end{theorem}

    Theorem \ref{teo:ESN-morphism} generalizes the ESN Theorem for Ehresmann semigroups \cite[Theorem~4.24]{lawson1991}. In fact, it follows from Lemma \ref{lema:ESN-semigrupo} that the category isomorphism of Theorem \ref{teo:ESN-morphism} restricts to the isomorphism between the categories of Ehresmann semigroups and biordered Ehresmann categories.

    On the other hand, by Theorem \ref{prop:ESN-objetos-restricao}, the object part of the functors from Theorem \ref{teo:ESN-morphism} restricts to a bijective correspondence between the objects of the categories $rSGPD_m$ and $liCAT_{if}$. Since the latter are full subcategories of $ESGPD_m$ and $lECAT_{if}$, respectively, the following result is an immediate consequence of Theorem \ref{teo:ESN-morphism}.

    \begin{corollary} \label{coro:ESN-morphism-r}
        The categories $rSGPD_m$ and $liCAT_{if}$ are isomorphic.
    \end{corollary}

    By Theorem \ref{prop:ESN-objetos-restricao} and Lemma \ref{lema:ESN-semigrupo}, restriction semigroups correspond to inductive categories. Hence, Corollary \ref{coro:ESN-morphism-r} generalizes the ESN Theorem for restriction semigroups \cite[Theorem 5.7]{lawson1991}, where restriction semigroups were referred to as \emph{idempotent-connected Ehresmann semigroups}.

    The following result is a consequence of Theorems \ref{prop:ESN-inverse} and \ref{teo:ESN-morphism}.

    \begin{corollary} \label{coro:ESN-morphism-i}
        The categories $iSGPD_m$ and $liGPD_{if}$ are isomorphic.
    \end{corollary}

    Note that any composition-preserving function between inverse semigroupoids is a (2,1,1)-morphism. Indeed, since pseudo-inverses in inverse semigroupoids are unique, any composition-preserving function necessarily preserves pseudo-inverses. Moreover, since in this setting one has $s^{+} = s^{-1}s$ and $s^{\ast} = ss^{-1}$, such functions also preserve the restriction operations. Thus, for inverse semigroupoids, composition-preserving functions are precisely the $(2,1,1)$-morphisms. In \cite{dewolf2018ehresmann} such functions are called \emph{semifunctors} and are the morphisms of the category of inverse semigroupoids (also called \emph{inverse semicategories}). Therefore, Corollary \ref{coro:ESN-morphism-i} coincides with \cite[Corollary 3.18]{dewolf2018ehresmann}. In particular, our result also covers the ``morphism part'' of the classical ESN Theorem for inverse semigroups \cite[Theorem 4.1.8]{lawson1998inverse}.

\subsection{\texorpdfstring{$\vee$}{}-Premorphisms and Ordered Functors} The notion of $\vee$-premorphism arose in the characterization of $E$-unitary covers for inverse semigroups. In particular, McAlister and Reilly showed that subdirect products of an inverse semigroup $S$ by a group $G$ can be characterized in terms of such functions \cite[Theorem 3.9]{mcalister1977}. On the other hand, the ESN Theorem for inverse semigroups \cite[Theorem 4.1.8]{lawson1998inverse} establishes that $\vee$-premorphisms between inverse semigroups correspond precisely to order-preserving functors between inductive groupoids. In this subsection, we extend the ``$\vee$-premorphism'' part of the ESN Theorem to the setting of restriction semigroupoids.\\

   Throughout what follows, $(S,+,\ast)$ and $(T,+,\ast)$ denote restriction semigroupoids with sets of projections $U$ and $V$, respectively, and $(\mathcal{C},\leq)$ and $(\mathcal{D},\leq)$ denote locally inductive categories.

    \begin{defi}
        A function $\varphi \colon S \to T$ between restriction semigroupoids is called a \emph{$\vee$-premorphism} if the following conditions are satisfied:
        \begin{enumerate} \Not{$\vee$}
            \item if $st$ is defined, then $\varphi(s)\varphi(t)$ is defined and $\varphi(st) \leq \varphi(s)\varphi(t)$; \label{vm1}
            \item $\varphi(s^+) = \varphi(s)^+$ and $\varphi(s^\ast) = \varphi(s)^\ast$, for every $s \in S$. \label{vm2}
        \end{enumerate}
    \end{defi}

    In \cite[Definition 4.2]{hollings2010extending}, a ``$\vee$-premorphism'' is defined as a function satisfying \eqref{vm1} together with the conditions $\varphi(s^+) \leq \varphi(s)^+$ and $\varphi(s^\ast) \leq \varphi(s)^\ast$, for every $s \in S$, which are more general than \eqref{vm2}. However, \cite[Lemma 4.5]{hollings2010extending} shows that such functions coincide precisely with the $\vee$-premorphisms we consider here. Accordingly, we adopt the formulation of $\vee$-premorphisms using equality in \eqref{vm2}.

    \begin{lemma} \label{lema:vpremorphism-order}
        If $\varphi \colon S \to Y$ is a $\wedge$-premorphism and $s \leq t$, then $\varphi(s) \leq \varphi(t)$.

        \begin{proof}
            Let $\varphi \colon S \to T$ be a $\vee$-premorphism and suppose that $s \leq t$. Then $s^+t$ is defined and $s^+t = s$. From \eqref{vm1} we obtain that $\varphi(s^+)\varphi(t)$ is defined and $\varphi(s) = \varphi(s^+t) \leq \varphi(s^+)\varphi(t)$. Hence,
                $$ \varphi(s) = \varphi(s)^+\varphi(s^+)\varphi(t) \overset{\eqref{vm2}}{=} \varphi(s)^+\varphi(s)^+\varphi(t) \overset{\ref{lema:le}(a)}{=} \varphi(s)^+\varphi(t). $$
            This proves that $\varphi(s) \leq \varphi(t)$.
        \end{proof}
    \end{lemma}

    \begin{lemma} \label{lemma:composition-vpremorphisms}
        The composition of $\vee$-premorphisms is a $\vee$-premorphism.
        
        \begin{proof}
            Let $\varphi \colon S \to T$ and $\psi \colon T \to R$ be $\vee$-premorphisms. Suppose that $st$ is defined. From \eqref{vm1} for $\varphi$, we obtain that $\varphi(s)\varphi(t)$ is defined and $\varphi(st) \leq \varphi(s)\varphi(t)$. Therefore, from \eqref{vm1} for $\psi$, we obtain that $\psi(\varphi(s))\psi(\varphi(t))$ is defined and $\psi(\varphi(s)\varphi(t)) \leq \psi(\varphi(s))\psi(\varphi(t))$. Hence,
                $$ \psi(\varphi(st)) \overset{\eqref{lema:vpremorphism-order}}{\leq} \psi(\varphi(s)\varphi(t)) \leq \psi(\varphi(s))\psi(\varphi(t)). $$
            That is, $\psi \circ \varphi$ satisfies \eqref{vm1}. Since $\varphi$ and $\psi$ satisfy \eqref{vm2}, for every $s \in S$, we obtain
                $$ \psi(\varphi(s^+)) = \psi(\varphi(s)^+) = \psi(\varphi(s))^+ \quad\text{and}\quad \psi(\varphi(s^\ast)) = \psi(\varphi(s)^\ast) = \psi(\varphi(s))^\ast. $$
            Therefore, $\psi \circ \varphi$ satisfies \eqref{vm2}. This proves that $\psi \circ \varphi$ is a $\vee$-premorphism.
        \end{proof}
    \end{lemma}

    From Lemma \ref{lemma:composition-vpremorphisms}, we obtain that restriction semigroupoids and $\vee$-premorphisms form a category, which we denote by $rSGPD_{\vee}$. We denote by $iSGPD_{\vee}$ the full subcategory of $rSGPD_{\vee}$ whose objects are inverse semigroupoids.

    \begin{defi}
        A function $\varphi \colon \mathcal{C} \to \mathcal{D}$ between locally inductive categories is called an \emph{ordered functor} if the following conditions are satisfied:
        \begin{enumerate} \Not{of}
            \item $\varphi(D(x)) = D(\varphi(x))$ and $\varphi(R(x)) = R(\varphi(x))$, for every $x \in \mathcal{C}$; \label{of1}
            \item if $D(x) = R(y)$, then $\varphi(x) \circ \varphi(y) = \varphi(x \circ y)$; \label{of2}
            \item if $x \leq y$, then $\varphi(x) \leq \varphi(y)$. \label{of3}
        \end{enumerate}
    \end{defi}

    \begin{lemma} \label{lema:ESN-vpremorphism}
        Let $\varphi \colon S \to T$ be a function. Regard $\varphi \colon C(S) \to C(T)$ as the same function. Then $\varphi \colon S \to T$ is a $\vee$-premorphisms if and only if $\varphi \colon C(S) \to C(T)$ is an ordered functor.

        \begin{proof}
            Suppose that $\varphi \colon S \to T$ is a $\vee$-premorphism. We prove that $\varphi$ is an ordered functor. Note that \eqref{of1} is precisely \eqref{vm2}, and \eqref{of3} is Lemma \ref{lema:vpremorphism-order}. Thus, it remains to prove that \eqref{of2} is satisfied. In fact, suppose that $s^\ast = D(s) = R(t) = t^+$. Then
            \begin{align*}
                \varphi(st) &= \varphi(st)^+ \varphi(s)\varphi(t) & \eqref{vm1} \\
                &= \varphi((st)^+) \varphi(s)\varphi(t) & \eqref{vm2} \\
                &= \varphi((st^+)^+) \varphi(s)\varphi(t) & \eqref{le4} \\
                &= \varphi(ss^\ast)^+ \varphi(s)\varphi(t) \\
                &= \varphi(s)^+ \varphi(s)\varphi(t) & \eqref{rr1} \\
                &= \varphi(s)\varphi(t). & \eqref{lr1}
            \end{align*}
            This proves that $\varphi(s \circ t) = \varphi(st) = \varphi(s)\varphi(t) = \varphi(s) \circ \varphi(t)$. Therefore, $\varphi \colon C(S) \to C(T)$ is an ordered functor. Conversely, suppose that $\varphi \colon C(S) \to C(T)$ is an ordered functor. We prove that $\varphi$ is a $\vee$-premorphism. Condition \eqref{vm2} is precisely \eqref{of1}. Thus, it remains to prove that \eqref{vm1} is satisfied.
            
            Suppose that $st$ is defined and let $e = s^\ast t^+ = D(s) \wedge R(t)$. Since $e \leq D(s)$, we have
                $$ \varphi(e) \overset{\eqref{of3}}{\leq} \varphi(D(s)) \overset{\eqref{of1}}{=} D(\varphi(s)). $$
            And analogously $\varphi(e) \leq R(\varphi(t))$. Since $\mathcal{D}_0$ is a local meet-semilattice, we have that $\varphi(e) \leq D(\varphi(s)) \wedge R(\varphi(t))$. Let $f = D(\varphi(s)) \wedge R(\varphi(t))$. Since $\varphi(e) \leq f \leq D(\varphi(s))$ and $\varphi(e) \leq f \leq R(\varphi(t))$, it follows from Lemma \ref{lema:restricao} that
                $$ \varphi(s)|\varphi(e) \leq \varphi(s)|f \quad\text{and}\quad \varphi(e)|\varphi(t) \leq f|\varphi(t). $$
            Since both $(\varphi(s)|\varphi(e)) \circ (\varphi(e)|\varphi(t))$ and $(\varphi(s)|f) \circ (f|\varphi(t))$ are defined, it follows that
                $$ (\varphi(s)|\varphi(e)) \circ (\varphi(e)|\varphi(t)) \overset{\eqref{O2}}{\leq} (\varphi(s)|f) \circ (f|\varphi(t)). $$
            On the other hand, since $s|e \leq s$, we have $\varphi(s|e) \leq \varphi(s)$ by \eqref{of3}, and $D(\varphi(s|e)) = \varphi(D(s|e)) = \varphi(D(e)) = \varphi(e)$ by \eqref{of1}. From the uniqueness of the restriction, we obtain $\varphi(s|e) = \varphi(s)|\varphi(e)$. Analogously, we obtain $\varphi(e|t) = \varphi(e)|\varphi(t)$. Therefore,
            \begin{align*}
                \varphi(s \otimes t) &= \varphi((s|e) \circ (e|t)) \\
                &= \varphi(s|e) \circ \varphi(e|t) & \eqref{of2} \\
                &= (\varphi(s)|\varphi(e)) \circ (\varphi(e)|\varphi(t)) & \eqref{restriction}, \eqref{corestriction} \\
                &\leq (\varphi(s)|f) \circ (f|\varphi(t)) & \eqref{O2} \\
                &= \varphi(s) \otimes \varphi(t).
            \end{align*}
            This proves that $\varphi(st) = \varphi(s)\varphi(t)$. Therefore, $\varphi \colon S \to T$ is a $\vee$-premorphism.
        \end{proof}
    \end{lemma}

    Analogously to Lemma \ref{coro:composition-ifunctors}, it follows from Lemma \ref{lema:ESN-vpremorphism} that the composition of ordered functors is again an ordered functor. Hence, locally inductive categories and ordered functors form a category, which we denote by $liCAT_{of}$. Furthermore, we denote by $liGPD_{of}$ the full subcategory of $liCAT_{of}$ whose objects are locally inductive groupoids.

    \begin{theorem} \label{teo:ESN-vpremorphism}
        The categories $rSGPD_{\vee}$ and $liCAT_{of}$ are isomorphic.

        \begin{proof}
            The proof is analogous to that of Theorem \ref{teo:ESN-morphism}. The correspondences $C$ and $S$ between restriction semigroupoids and locally inductive categories extend to category isomorphisms between $rSGPD_{\vee}$ and $liCAT_{of}$.
        \end{proof}
    \end{theorem}

    In the same way as in Corollary \ref{coro:ESN-morphism-r}, we conclude that Theorem \ref{teo:ESN-vpremorphism} generalizes the ESN-type Theorem for restriction semigroups and $\vee$-premorphisms \cite[Theorem 4.1]{hollings2010extending}. 
    
    The following result is a consequence of Theorem \ref{teo:ESN-vpremorphism} and Lemmas \ref{lema:inverse-1} and \ref{lema:inverse-2}.

    \begin{corollary} \label{coro:ESN-vpremorphism-i}
        The categories $iSGPD_{\vee}$ and $liGPD_{of}$ are isomorphic.
    \end{corollary}

    An inverse category is a category $\mathcal{C}$ such that, for every arrow $f\in\mathcal{C}$, there exists a unique arrow $g\in\mathcal{C}$ such that $fg$ and $gf$ are defined and $fgf=f$ and $gfg=g$.

    In \cite[Definition 3.20]{dewolf2018ehresmann}, a function $F \colon S \to T$ between inverse categories was called an \emph{oplax functor} if $F(1_A) \leq 1_{F(A)}$ and, whenever $g \circ f$ is defined, $Fg \circ Ff$ is also defined and $F(g \circ f) \leq Fg \circ Ff$. Similar to \cite[Theorem 3.1.5]{lawson1998inverse}, one can prove that oplax functors preserve pseudo-inverses. Hence, the proof of \cite[Lemma 4.4]{hollings2010extending} shows that oplax functors are precisely the $\vee$-premorphisms between inverse categories. On the other hand, it follows from Theorem \ref{prop:ESN-inverse} and Lemma \ref{lema:ESN-cat} that inverse categories correspond to locally complete inductive groupoids. Hence, Corollary \ref{coro:ESN-vpremorphism-i} generalizes the ESN Theorem for inverse categories and oplax functors \cite[Theorem 3.21]{dewolf2018ehresmann}. In particular, we also obtain the ``premorphism part'' of the classical Ehresmann-Schein-Nambooripad Theorem for inverse semigroups \cite[Theorem 4.1.8]{lawson1998inverse}.

    \begin{obs} \label{obs:join}
        (1) Every (2,1,1)-morphism is a $\vee$-premorphism. In fact, condition \eqref{m2} is equivalent to \eqref{vm2}, while $\varphi(st) = \varphi(s)\varphi(t)$ is a particular case of $\varphi(st) \leq \varphi(s)\varphi(t)$. Hence, condition \eqref{m1} implies \eqref{vm1}. Therefore, the category isomorphism $C \colon rSGPD_{\vee} \to liCAT_{of}$ is an extension of the category isomorphism $C \colon rSGPD_m \to liCAT_{if}$.

        (2) At present, it is unclear whether the category isomorphism
        $C \colon ESGPD_m \to lECAT_{if}$ can be extended so as to generalize the isomorphism $C \colon rSGPD_{\vee} \to liCAT_{of}$. This difficulty arises from the use of \eqref{O2} in the proof that, if $\varphi \colon C(S) \to C(T)$ is an ordered functor, then $\varphi \colon S \to T$ is a $\vee$-premorphism. For if $S$ is an Ehresmann semigroupoid, then we can obtain $\varphi(s)|\varphi(e) \leq_r \varphi(s)|f$ and $\varphi(e)|\varphi(t) \leq_l f|\varphi(t)$, but since the partial orders may not coincide, we can not use \eqref{O2} to conclude that $\varphi(st) = \varphi(s)|\varphi(e) \circ \varphi(e)|\varphi(t) \leq \varphi(s)|f \circ f|\varphi(t) = \varphi(s)\varphi(t)$.
    \end{obs}

\subsection{\texorpdfstring{$\wedge$}{}-Premorphisms and Inductive Prefunctors} The notion of $\wedge$-premorphisms was also introduced by McAlister and Reilly, who showed that every $E$-unitary cover of an inverse semigroup can be constructed via such functions \cite[Theorem 4.5]{mcalister1977}. More recently, $\wedge$-premorphisms have appeared in the study of partial actions (see, for instance, \cite[Proposition 2.10]{hollings2007monoids}) and provide a fundamental characterization of the Szendrei expansion of one-sided restriction semigroupoids \cite[Theorem 4.9]{rsgpdexpansion}. In this subsection, we extend the ``$\wedge$-premorphism'' part of the ESN Theorem for restriction semigroups \cite[Theorem 5.9]{hollings2010extending} to restriction semigroupoids.

    \begin{defi}
        A function $\varphi \colon S \to T$ between restriction semigroupoids is called a \emph{$\wedge$-premorphism} if the following conditions are satisfied:
        \begin{enumerate} \Not{$\wedge$}
            \item if $st$ is defined, then $\varphi(s)\varphi(t)$, $\varphi(st)\varphi(t)^\ast$ and $\varphi(s)^+\varphi(st)$ are defined and \label{wm1}
                $$ \varphi(s)\varphi(t) = \varphi(s)^+\varphi(st) = \varphi(st)\varphi(t)^\ast; $$
            \item $\varphi(s)^+ \leq \varphi(s^+)$ and $\varphi(s)^\ast \leq \varphi(s^\ast)$, for all $s \in S$. \label{wm2}
        \end{enumerate}
    \end{defi}

    \begin{lemma} \label{lema:wpremorphisms-order}
        Let $\varphi \colon S \to T$ be a $\wedge$-premorphism.
        \begin{itemize}
            \item[(a)] If $e \in U$, then $\varphi(e) \in V$.
            \item[(b)] If $s \leq t$, then $\varphi(s) \leq \varphi(t)$.
        \end{itemize}

        \begin{proof}
            (a) Let $e \in U$. Then $e = e^+$ by Lemma \ref{lema:le}(a). Therefore,
                $$ \varphi(e) \overset{\eqref{lr1}}{=} \varphi(e)^+ \varphi(e) = \varphi(e)^+ \varphi(e^+) \overset{\eqref{wm2}}{=} \varphi(e)^+ \in V. $$
            (b) First, suppose that $e,f \in U$ are such that $e \leq f$. Then $ef$ is defined and $e = ef$. From \eqref{wm1}, we obtain that $\varphi(e)\varphi(f)$ is defined and
                $$ \varphi(e)\varphi(f) \overset{\eqref{wm1}}{=} \varphi(e)^+ \varphi(ef) = \varphi(e)^+ \varphi(e) \overset{\eqref{lr1}}{=} \varphi(e). $$
            That is, $\varphi(e) \leq \varphi(f)$. Now, suppose that $s \leq t$. Then there is $e \in U$ such that $et$ is defined and $et = s$. From \eqref{wm1}, we obtain that $\varphi(e)\varphi(t)$ is defined and
                $$ \varphi(e)\varphi(t) \overset{\eqref{wm1}}{=} \varphi(e)^+ \varphi(et) = \varphi(e)^+ \varphi(s) \overset{(a)}{=} \varphi(e)\varphi(s). $$
            On the other hand, it follows from Lemma \ref{lema:le}(b) that $s^+ = (et)^+ = (et)^+ e^+ \leq e^+ = e$. Since $s^+,e \in U$, we conclude from the previous argument that $\varphi(s)^+ \leq \varphi(e)$. That is, $\varphi(s)^+ = \varphi(s)^+ \varphi(e)$. Therefore, $\varphi(s)^+ \varphi(t)$ is defined and
                $$ \varphi(s)^+\varphi(t) = \varphi(s)^+ \varphi(e)\varphi(t) = \varphi(s)^+ \varphi(e) \varphi(s) = \varphi(s)^+ \varphi(s) \overset{\eqref{lr1}}{=} \varphi(s). $$
            This proves that $\varphi(s) \leq \varphi(t)$.
        \end{proof}
    \end{lemma}

    \begin{lemma} \label{lemma:composition-wpremorphism}
        The composition of $\wedge$-premorphisms is a $\wedge$-premorphism.

        \begin{proof}
            Let $\varphi \colon S \to T$ and $\psi \colon T \to R$ be $\wedge$-premorphisms. From \eqref{wm2} for $\varphi$, we obtain that $\varphi(s)^+ \leq \varphi(s^+)$, for all $s \in S$. Hence,
                $$ \psi(\varphi(s))^+ \overset{\eqref{wm2}}{\leq} \psi(\varphi(s)^+) \overset{\ref{lema:wpremorphisms-order}(b)}{\leq} \psi(\varphi(s^+)). $$
            Analogously, we have $\psi(\varphi(s))^\ast \leq \psi(\varphi(s^\ast))$. Therefore, $\psi \circ \varphi$ satisfies \eqref{wm2}. Suppose that $st$ is defined. From \eqref{wm1} for $\varphi$, we obtain that $\varphi(s)\varphi(t)$ is defined and $\varphi(s)\varphi(t) = \varphi(s)^+\varphi(st)$. Thus, from \eqref{wm1} for $\psi$, we have that $\psi(\varphi(s))\psi(\varphi(t))$ is defined and
                $$ \psi(\varphi(s)) \psi(\varphi(t)) = \psi(\varphi(s))^+ \psi(\varphi(s)\varphi(t)) = \psi(\varphi(s))^+ \psi(\varphi(s)^+ \varphi(st)). $$
            On the other hand, let $x,y \in R$ and suppose that $x^+ \leq y$. Then $x^+ = (x^+y)^+ = x^+y^+$ by \eqref{lr3}. That is, $x^+ \leq y^+$. Let $x^+ = \psi(\varphi(s))^+$ and $y = \psi(\varphi(s)^+)$. We have $x^+ \leq y$ by \eqref{wm2} for $\varphi$. Thus,
            \begin{align*}
                \psi(\varphi(s)) \psi(\varphi(t)) &= \psi(\varphi(s))^+ \psi(\varphi(s)^+ \varphi(st)) \\
                &= \psi(\varphi(s))^+ \psi(\varphi(s)^+)^+ \psi(\varphi(s)^+ \varphi(st)) & (x^+ \leq y^+) \\
                &= \psi(\varphi(s))^+ \psi(\varphi(s)^+) \psi(\varphi(st)) & \eqref{wm1} \\
                &= \psi(\varphi(s))^+ \psi(\varphi(st)). & (x^+ \leq y)
            \end{align*}
            Analogously, we have $\psi(\varphi(s))\psi(\varphi(t)) = \psi(\varphi(st))\psi(\varphi(t))^\ast$. Hence, $\psi \circ \varphi$ satisfies \eqref{wm1}.
        \end{proof}
    \end{lemma}

    From Lemma \ref{lemma:composition-wpremorphism}, we obtain that restriction semigroupoids and $\wedge$-premorphisms form a category, which we denote by $rSGPD_{\wedge}$. Furthermore, we denote by $iSGPD_{\wedge}$ the full subcategory of $rSGPD_{\wedge}$ whose objects are inverse semigroupoids. Note that $rSGPD_m$ is a subcategory of the category $rSGPD_{\wedge}$. In fact, if $\varphi$ is a (2,1,1)-morphism between restriction semigroupoids, then
        $$ \varphi(s)^+ \varphi(st) \overset{\eqref{m1}}{=} \varphi(s)^+ \varphi(s)\varphi(t) \overset{\eqref{lr1}}{=} \varphi(s)\varphi(t), $$
    and analogously $\varphi(st) \varphi(t)^\ast = \varphi(s)\varphi(t)$. Hence, \eqref{wm1} holds, and since \eqref{m2} is stronger than \eqref{wm2}, every $(2,1,1)$-morphism is a $\wedge$-premorphism.

    \begin{defi}
        A function $\varphi \colon \mathcal{C} \to \mathcal{D}$ between locally inductive categories is called an \emph{inductive prefunctor} if the following conditions are satisfied:
        \begin{enumerate} \Not{ip}
            \item if $x \circ y$ is defined, then $\varphi(x) \otimes \varphi(y) \leq \varphi(x \circ y)$; \label{ip1}
            \item $D(\varphi(x)) \leq \varphi(D(x))$ and $R(\varphi(x)) \leq \varphi(R(x))$, for every $x \in \mathcal{C}$; \label{ip2}
            \item if $x \leq y$, then $\varphi(x) \leq \varphi(y)$; \label{ip3}
            \item if $e,f \in \mathcal{C}_0$ are such that $x \otimes e$ and $f \otimes x$ are defined, then \label{ip4}
                $$ \varphi(x) \otimes \varphi(e) \leq \varphi(x \otimes e) \quad\text{and}\quad \varphi(f) \otimes \varphi(x) \leq \varphi(f \otimes x); $$
            \item if $x \otimes y$ is defined, then \label{ip5}
                $$ R(\varphi(x) \otimes \varphi(y)) = R(\varphi(x)) \wedge R(\varphi(x \otimes y)), $$
            and
                $$ D(\varphi(x) \otimes \varphi(y)) = D(\varphi(x \otimes y)) \wedge D(\varphi(y)). $$
        \end{enumerate}
    \end{defi}

    \begin{obs} \label{obs:ip}
        The axioms \eqref{ip1}, \eqref{ip4} and \eqref{ip5} are well defined. First, note that condition \eqref{ip2} implies that $\varphi(\mathcal{C}_0) \subseteq \mathcal{D}_0$. Indeed, if $e \in \mathcal{C}_0$, then $e = R(e)$. Hence, by \eqref{ip2}, we have $R(\varphi(e)) \leq \varphi(R(e)) = \varphi(e)$. Therefore,
            $$ R(\varphi(e)) = R(\varphi(e))^+ \otimes \varphi(e) = \varphi(e)^+ \otimes \varphi(e) \overset{\eqref{lr1}}{=} \varphi(e). $$
        This shows that $\varphi(e) = R(\varphi(e)) \in \mathcal{D}_0$. 
        
        We now show that if $x \otimes y$ is defined, then $\varphi(x) \otimes \varphi(y)$ is defined. Since $x \circ y = x \otimes y$ whenever $x \circ y$ is defined, this proves that \eqref{ip1} and \eqref{ip4} are well defined. Note that if $e,f,g \in \mathcal{D}_0$ are such that $e \wedge f$ and $g \wedge f$ are defined, then $e \wedge g$ is defined, since $\mathcal{D}_0$ is a local meet-semilattice.
        
        Assume that $x \otimes y$ is defined. Then $e = D(x) \wedge R(y)$ is defined. Since $e \leq D(x)$, condition \eqref{ip3} yields $\varphi(e) \leq \varphi(D(x))$, and hence $\varphi(e) \wedge \varphi(D(x))$ is defined. Moreover, by \eqref{ip2}, we have $D(\varphi(x)) \leq \varphi(D(x))$, and therefore $D(\varphi(x)) \wedge \varphi(D(x))$ is defined. Then $\varphi(e) \wedge D(\varphi(x))$ is defined. Analogously, one obtains that $\varphi(e) \wedge R(\varphi(y))$ is defined. Consequently, $D(\varphi(x)) \wedge R(\varphi(y))$ is defined, and hence $\varphi(x) \otimes \varphi(y)$ is defined.

        Finally, to show that \eqref{ip5} is well defined, assume again that $x \otimes y$ is defined. Then $e = D(x) \wedge R(y)$ is defined. On the one hand, we have
            $$ D(\varphi(x \otimes y)) \overset{\eqref{ip2}}{\leq} \varphi(D(x \otimes y)) = \varphi(D((x|e) \circ (e|y))) = \varphi(D(e|y)). $$
        Since $e|y \leq y$, we obtain $D(e|y) \leq D(y)$. Hence, $D(\varphi(x \otimes y)) \leq \varphi(D(e|y)) \leq \varphi(D(y))$ by \eqref{ip3}. On the other hand, we have $D(\varphi(x)) \leq \varphi(D(x))$ by \eqref{ip2}. That is,
            $$ D(\varphi(x \otimes y)) \wedge \varphi(D(y)) \quad\text{and}\quad D(\varphi(y)) \wedge \varphi(D(y)) $$
        are defined. Thus, $D(\varphi(x \otimes y)) \wedge D(\varphi(x))$ is defined. By a symmetric argument, $R(\varphi(x)) \wedge R(\varphi(x \otimes y))$ is defined. This shows that \eqref{ip5} is well defined.
    \end{obs}

    \begin{lemma} \label{lema:ESN-wpremorphism}
        Let $\varphi \colon S \to T$ be a function. Regard $\varphi \colon C(S) \to C(T)$ as the same function. Then $\varphi \colon S \to T$ is a $\wedge$-premorphism if and only if $\varphi \colon C(S) \to C(T)$ is an inductive prefunctor.

        \begin{proof}
            Suppose that $\varphi \colon S \to T$ is a $\wedge$-premorphism. We prove that $\varphi$ is an inductive prefunctor. Notice that \eqref{ip2} is precisely \eqref{wm2}, and \eqref{ip3} is Lemma \ref{lema:wpremorphisms-order}(b). Suppose that $s \otimes t = st$ is defined. From \eqref{wm1}, we obtain that
                $$ \varphi(s) \otimes \varphi(t) = \varphi(s)\varphi(t) = \varphi(s)^+ \varphi(s \otimes t) \leq \varphi(s \otimes t). $$
            In particular, taking $s \in \mathcal{C}_0$ or $t \in \mathcal{C}_0$, we conclude \eqref{ip4}. If $s \circ t$ is defined, then $s \otimes t$ is defined and $s \otimes t = s \circ t$. Therefore,
                $$ \varphi(s) \otimes \varphi(t) \leq \varphi(s \otimes t) = \varphi(s \circ t). $$
            That is, \eqref{ip1} is satisfied. Lastly, if $s \otimes t$ is defined, then
            \begin{align*}
                D(\varphi(s) \otimes \varphi(t)) &= (\varphi(s)\varphi(t))^\ast \\
                &= (\varphi(st) \varphi(t)^\ast)^\ast & \eqref{wm1} \\
                &= \varphi(st)^\ast \varphi(t)^\ast & \eqref{rr3} \\
                &= D(\varphi(s \otimes t)) \wedge D(\varphi(t)).
            \end{align*}
            Analogously, we have that $R(\varphi(s) \otimes \varphi(t)) = R(\varphi(s)) \wedge R(\varphi(s \otimes t))$. This proves that $\varphi \colon C(S) \to C(T)$ is an inductive prefunctor. Conversely, suppose that $\varphi$ is an inductive prefunctor. We prove that $\varphi$ is a $\wedge$-premorphism. Condition \eqref{wm2} is precisely \eqref{ip2}. It remains to prove \eqref{wm1}.
            
            Suppose that $st = s \otimes t$ is defined. From Remark \ref{obs:ip}, we obtain that $\varphi(s) \otimes \varphi(t)$ is defined. Let $e = D(s) \wedge R(t)$ and $f = D(\varphi(s)) \wedge R(\varphi(t))$. From \eqref{ip2} and the fact that $C(T)_0$ is a local meet-semilattice, we have that
                $$ f \leq D(\varphi(s)) \wedge \varphi(R(t)) \leq D(\varphi(s)) \quad \text{and}\quad f \leq \varphi(D(s)) \wedge R(\varphi(t)) \leq R(\varphi(t)). $$
            Therefore, it follows from Lemma \ref{lema:restricao} that 
                $$ \varphi(s)|f \leq \varphi(s)|( D(\varphi(s)) \wedge \varphi(R(t)) ) = \varphi(s) \otimes \varphi(R(t)), $$
            and
                $$ f|\varphi(t) \leq (\varphi(D(s)) \wedge R(\varphi(t)))|\varphi(t) = \varphi(D(s)) \otimes \varphi(t). $$
            On the other hand, let $x,x',y,y' \in T$ be such that $xy$ is defined, $x \leq x'$ and $y \leq y'$. Since $T$ is a restriction semigroupoid, we have that $x'y'$ is defined and $xy \leq x'y'$. In fact,
                $$ xy = (x^+x')(y^+y') = x^+(x'y^+)y' \overset{\eqref{lr4}}{=} x^+((x'y)^+ x')y' = (x^+(x'y)^+) (x'y') \leq x'y'. $$
            From the previous observation, we obtain
            \begin{align*}
                \varphi(s) \otimes \varphi(t) &= (\varphi(s)|f) \otimes (f|\varphi(t)) \\
                &\leq (\varphi(s) \otimes \varphi(R(t))) \otimes (\varphi(D(s)) \otimes \varphi(t)) \\
                &\leq \varphi(s \otimes R(t)) \otimes \varphi(D(s) \otimes t) & \eqref{ip4} \\
                &= \varphi(s|e) \otimes \varphi(e|t) \\
                &\leq \varphi((s|e) \circ (e|t)) & \eqref{ip1} \\
                &= \varphi(s \otimes t).
            \end{align*}
            Hence,
            \begin{align*}
                \varphi(s) \otimes \varphi(t) &= R(\varphi(s) \otimes \varphi(t)) \otimes \varphi(s \otimes t) \\
                &= (R(\varphi(s)) \wedge R(\varphi(s \otimes t))) \otimes \varphi(s \otimes t) & \eqref{ip5} \\
                &= \varphi(s)^+ \otimes \varphi(s \otimes t)^+ \otimes \varphi(s \otimes t) \\
                &= \varphi(s)^+ \otimes \varphi(s \otimes t). & \eqref{lr1}
            \end{align*}
            This proves that, if $st$ is defined, then $\varphi(s)\varphi(t) = \varphi(s)^+ \varphi(st)$. Analogously, $\varphi(s)\varphi(t) = \varphi(st)\varphi(t)^\ast$. Therefore, $\varphi \colon S \to T$ is a $\wedge$-premorphism.
        \end{proof}
    \end{lemma}

    Analogously to Lemma \ref{coro:composition-ifunctors}, it follows from Lemma \ref{lema:ESN-wpremorphism} that the composition of inductive prefunctors is an inductive prefunctor. Hence, locally inductive categories and $\wedge$-premorphisms form a category, which we denote by $liCAT_{ip}$. Furthermore, we denote by $liGPD_{ip}$ the full subcategory of $liCAT_{ip}$ whose objects are locally inductive groupoids.

    \begin{theorem} \label{teo:ESN-wpremorphism}
        The categories $rSGPD_{\wedge}$ and $liCAT_{ip}$ are isomorphic.

        \begin{proof}
            The proof is analogous to that of Theorem \ref{teo:ESN-morphism}. The correspondences $C$ and $S$ between restriction semigroupoids and locally inductive categories extend to categories isomorphisms between $rSGPD_{\wedge}$ and $liCAT_{ip}$.
        \end{proof}
    \end{theorem}

    From Theorem \ref{prop:ESN-objetos-restricao} and Lemma \ref{lema:ESN-semigrupo}, restriction semigroups correspond to inductive groupoids. Hence, Theorem \ref{teo:ESN-wpremorphism} generalizes the ESN Theorem for restriction semigroups and $\wedge$-premorphisms \cite[Theorem 5.9]{hollings2010extending}. Moreover, it follows from Theorem \ref{prop:ESN-inverse} that inverse semigroupoids correspond to locally inductive groupoids. Thus, the following result is an immediate consequence of Theorem \ref{teo:ESN-wpremorphism}.

    \begin{corollary} \label{coro:ESN-wpremorphism-i}
        The categories $iSGPD_{\wedge}$ and $liGPD_{ip}$ are isomorphic.
    \end{corollary}

     Corollary \ref{coro:ESN-wpremorphism-i} generalizes the ESN Theorem for inverse semigroups and $\wedge$-premorphisms \cite[Theorem 6.1]{hollings2010extending}. 
     
     From Theorems \ref{teo:ESN-morphism}, \ref{teo:ESN-vpremorphism} and \ref{teo:ESN-wpremorphism}, we obtain the complete version of the ESN Theorem for two-sided restriction semigroupoids, as stated below.

    \begin{theorem}\label{theo:esnprincipal}
        \begin{itemize}
            \item[(a)] The category of restriction semigroupoids and (2,1,1)-morphisms is isomorphic to the category of locally inductive categories and inductive functors.

            \item[(b)] The category of restriction semigroupoids and $\vee$-premorphisms is isomorphic to the category of locally inductive categories and ordered functors.

            \item[(c)] The category of restriction semigroupoids and $\wedge$-premorphisms is isomorphic to the category of locally inductive categories and inductive prefunctors.
        \end{itemize}
    \end{theorem}

    We end this section summarizing the conclusions of Theorem \ref{theo:esnprincipal} together with Corollaries \ref{coro:ESN-morphism-r}, \ref{coro:ESN-morphism-i}, \ref{coro:ESN-vpremorphism-i} and \ref{coro:ESN-wpremorphism-i} in a correspondence of lattices of categories. The arrows $C$ and $S$ in the diagram below represent the constructions $C$ and $S$ on Theorem \ref{prop:ESN-objetos}.

    \begin{center}
        \begin{tikzpicture}
            \tikzstyle{every path}=[draw];

            \node (Espd-m) at (0,0) {$ESGPD_m$};
            \node (Rspd-m) at (0,-2) {$rSGPD_m$};
            \node (Ispd-m) at (0,-4) {$iSGPD_m$};
            \node (Rspd-v) at (2,-1) {$rSGPD_\vee$};
            \node (Ispd-v) at (2,-3) {$iSGPD_\vee$};
            \node (Rspd-w) at (-2,-1) {$rSGPD_\wedge$};
            \node (Ispd-w) at (-2,-3) {$iSGPD_\wedge$};

            \path (Ispd-m) to (Ispd-v); \path (Ispd-m) to (Ispd-w); \path (Ispd-m) to (Rspd-m);
            \path (Rspd-m) to (Rspd-v); \path (Rspd-m) to (Rspd-w); \path (Rspd-m) to (Espd-m);
            \path (Rspd-v) to (Ispd-v); \path (Rspd-w) to (Ispd-w);

            \node (Espd-m) at (0+8,0) {$lECAT_{if}$};
            \node (Rspd-m) at (0+8,-2) {$liCAT_{if}$};
            \node (Ispd-m) at (0+8,-4) {$liGPD_{if}$};
            \node (Rspd-v) at (2+8,-1) {$liCAT_{of}$};
            \node (Ispd-v) at (2+8,-3) {$liGPD_{of}$};
            \node (Rspd-w) at (-2+8,-1) {$liCAT_{ip}$};
            \node (Ispd-w) at (-2+8,-3) {$liGPD_{ip}$};

            \path (Ispd-m) to (Ispd-v); \path (Ispd-m) to (Ispd-w); \path (Ispd-m) to (Rspd-m);
            \path (Rspd-m) to (Rspd-v); \path (Rspd-m) to (Rspd-w); \path (Rspd-m) to (Espd-m);
            \path (Rspd-v) to (Ispd-v); \path (Rspd-w) to (Ispd-w);

            \draw[->] (3,-1.4-0.5) -- (5,-1.4-0.5) node[midway, above]{$C$};
            \draw[<-] (3,-1.6-0.5) -- (5,-1.6-0.5) node[midway, below]{$S$};
        \end{tikzpicture}
    \end{center}

    Due to Proposition \ref{lema:ESN-semigrupo}, the above diagram is still valid if we replace semigroupoids by semigroups, and the corresponding local categories by their non-local version (when $\mathcal{C}_0$ is a meet-semilattice). In the same way, it follows from Proposition \ref{lema:ESN-cat} that the above diagram is still valid if we replace semigroupoids by categories, and the corresponding local categories by their locally complete version (when $\mathcal{C}_0$ is a locally complete meet-semilattice).

    \section*{Declarations}
    
    \subsection*{Author's Contribution}
    All authors wrote the main manuscript text and reviewed the manuscript.
    
    \subsection*{Competing Interests}
    The authors have no competing interests as defined by Springer, or other interests that might be perceived to influence the results and/or discussion reported in this paper.
    
    \subsection*{Ethical Approval}
    Not applicable.
    
    \subsection*{Availability of Data and Materials}
    Data sharing not applicable to this article as no datasets were generated or analysed during the current study.
    
    \subsection*{Funding}
    T. Tamusiunas was partially supported by Conselho Nacional de Desenvolvimento Científico e Tecnológico (Brazil) through a Productivity Research Fellowship, Grant No. 303411/2025-2, and by Conselho Nacional de Desenvolvimento Científico e Tecnológico (Brazil) under the Universal Call, Grant No. 403606/2025-0.

    \bibliographystyle{abbrvnat}
    \footnotesize{\bibliography{ref}}
\end{document}